\RequirePackage{amsthm}

\documentclass[sn-mathphys,Numbered,pdflatex]{sn-jnl}
\usepackage[defaultsans]{droidsans}
\usepackage{lmodern}
\usepackage{anyfontsize}

\usepackage{amsmath,amssymb,amsfonts}
\usepackage{amsthm}
\usepackage{mathrsfs}
\usepackage{bm}
\usepackage{mathtools}
\usepackage{hyperref}
\hypersetup{
	colorlinks,
	linkcolor={blue!30!black},
	citecolor={blue!50!black},
	urlcolor={blue!80!black}
}
\usepackage[title]{appendix}
\usepackage{textcomp}
\usepackage{manyfoot}
\usepackage{listings}

\usepackage{booktabs}
\makeatletter
\def\caption@documentclass{elsarticle}%
\makeatother
\usepackage[hang,small,bf]{caption}
\usepackage[scale=2]{ccicons}
\usepackage{colortbl,color}
\usepackage{pgfplots}
\pgfplotsset{compat=newest}
\usepgfplotslibrary{dateplot}
\usepackage{graphicx} 
\graphicspath{{figures/}}
\usepackage{subcaption}
\usepackage{xcolor,xspace}
\usepackage{here}
\usepackage{rotating}
\usepackage{multirow}
\usepackage{tablefootnote}

\usepackage{tikz}
\usetikzlibrary{arrows.meta,bbox}
\usetikzlibrary{calc}

\usepackage{bibunits}
\usepackage{changepage}
\usepackage{enumerate}
\usepackage{framed}
\usepackage{tcolorbox}
\usepackage{hyperref}

\newcommand{\zero}{\bm{0}}
\newcommand{\one}{\bm{1}}
\renewcommand{\a}{\bm{a}}
\renewcommand{\b}{\bm{b}}

\renewcommand{\d}{\bm{d}}

\renewcommand{\u}{\bm{u}}

\newcommand{\x}{\bm{x}}
\newcommand{\y}{\bm{y}}

\newcommand{\A}{\bm{A}}
\newcommand{\B}{\bm{B}}
\renewcommand{\H}{\bm{H}}
\newcommand{\I}{\bm{I}}

\newcommand{\Q}{\bm{Q}}
\newcommand{\W}{\bm{W}}

\DeclareMathOperator{\cl}{cl}
\DeclareMathOperator{\interior}{int}
\DeclareMathOperator{\R}{\mathbb{R}}

\DeclareMathOperator{\dom}{dom}
\newcommand{\extended}{[-\infty,+\infty]}
\newcommand{\extendedleft}{(-\infty,+\infty]}

\DeclareMathOperator{\diag}{diag}
\DeclareMathOperator{\dist}{dist}

\DeclareMathOperator{\prox}{prox}
\DeclareMathOperator{\sgn}{sgn}
\DeclareMathOperator{\soft}{soft}

\DeclareMathOperator*{\argmin}{argmin}

\newcommand{\ie}{i.e.}
\newcommand{\etal}{et al.}

\numberwithin{equation}{section}
\theoremstyle{thmstyleone}
\newtheorem{theorem}{Theorem}
\newtheorem{proposition}[theorem]{Proposition}
\newtheorem{lemma}[theorem]{Lemma}

\theoremstyle{thmstyletwo}
\newtheorem{example}[theorem]{Example}
\newtheorem{remark}[theorem]{Remark}
\newtheorem{assumption}[theorem]{Assumption}

\theoremstyle{thmstylethree}
\newtheorem{definition}[theorem]{Definition}

\usepackage[noend]{algpseudocode}
\usepackage{algorithm}
\usepackage{algorithmicx}

\algnewcommand{\Input}[1]{
\State\textbf{Input:}\hspace*{0.3em}\parbox[t]{.9\linewidth}{\raggedright #1}
}
\algnewcommand{\Initialization}[1]{
\State\textbf{Initialization:}\hspace*{0.3em}\parbox[t]{.8\linewidth}{\raggedright #1}
}

\title{Approximate Bregman Proximal Gradient Algorithm for  Relatively Smooth Nonconvex Optimization}

\author*[1]{\fnm{Takahashi} \sur{Shota}}\email{shota@mist.i.u-tokyo.ac.jp}

\author[1,2]{\fnm{Takeda} \sur{Akiko}}\email{takeda@mist.i.u-tokyo.ac.jp}

\affil[1]{\orgdiv{Graduate School of Information Science and Technology}, \orgname{The University of Tokyo}, \orgaddress{\street{7--3--1 Hongo}, \city{Bunkyo-ku}, \postcode{113--8656}, \state{Tokyo}, \country{Japan}}}

\affil[2]{\orgdiv{Center for Advanced Intelligence Project}, \orgname{RIKEN}, \orgaddress{\street{1--4--1 Nihonbashi}, \city{Chuo-ku}, \postcode{103--0027}, \state{Tokyo}, \country{Japan}}}

\date{\today}

\begin{document}

\abstract{
In this paper, we propose the approximate Bregman proximal gradient algorithm (ABPG) for solving composite nonconvex optimization problems. 
ABPG employs a new distance that approximates the Bregman distance, making the subproblem of ABPG simpler to solve compared to existing Bregman-type algorithms. 
The subproblem of ABPG is often expressed in a closed form.
Similarly to existing Bregman-type algorithms, ABPG does not require the global Lipschitz continuity for the gradient of the smooth part.
Instead, assuming the smooth adaptable property, we establish the global subsequential convergence under standard assumptions.
Additionally, assuming that the Kurdyka--\L ojasiewicz property holds, we prove the global convergence for a special case.
Our numerical experiments on the $\ell_p$ regularized least squares problem, the $\ell_p$ loss problem, and the nonnegative linear system show that ABPG outperforms existing algorithms especially when the gradient of the smooth part is not globally Lipschitz or even locally Lipschitz continuous.}

\keywords{
composite nonconvex nonsmooth optimization, Bregman proximal gradient algorithms, Kurdyka--\L ojasiewicz property, $\ell_p$ regularized least squares problem, $\ell_p$ loss problem, nonnegative linear system, Kullback--Leibler divergence
}

\pacs[MSC Classification]{90C26, 49M37, 65K05}

\maketitle

\section{Introduction}
In applications to machine learning and signal processing, regularization is used to avoid over-fitting or to impose the model structure.
An optimization problem with a regularization term can be formulated as a composite nonconvex optimization problem:
\begin{align}
    \label{prob:minimization}
    \min_{\x\in\cl{C}}\quad \Psi(\x) := f(\x) + g(\x),
\end{align}
where $f:\R^n\to\extendedleft$ is a continuously differentiable function, $g:\R^n\to\extendedleft$ is a convex function, and $\cl{C}$ denotes the closure of $C\subset\R^n$ that is a nonempty open convex set (see Assumption~\ref{assu:function} and Remark~\ref{remark:effective-domain} for a more precise statement).
In the context of machine learning or signal processing, $g$ is called a regularizer or a penalty function.
Problem~\eqref{prob:minimization} includes many applications, such as a maximum a posteriori probability (MAP) estimate~\cite{Bouman1993-ty,Elad2010-hk}, ridge regression~\cite{Hastie2009-sb}, the least absolute shrinkage and selection operator (LASSO)~\cite{Tibshirani1996-we}, phase retrieval~\cite{Bolte2018-zt,Takahashi2022-ml}, and blind deconvolution~\cite{Chan2000-os,Li2019-ab,Takahashi2023-uh}.

The most well-known algorithm for solving~\eqref{prob:minimization} would be a proximal algorithm.
The proximal algorithm uses the proximal mapping for its updates, that can be easily computed for a certain $g$ (see, for example,~\cite{Beck2017-qc}) in a closed-form expression.
The proximal algorithm has been actively studied since the 1970s (see, e.g., \cite{Rockafellar1976-th,Fukushima1981-wn});
The proximal algorithm includes many algorithms such as the proximal gradient method~\cite{Bruck1975-et,Lions1979-id,Passty1979-sk}, the fast iterative shrinkage-thresholding algorithm (FISTA)~\cite{Beck2009-kr}, the proximal Newton method~\cite{Patriksson1999-yp}, and the proximal quasi-Newton method~\cite{Becker2012-ji,Lee2014-zv}.
On the other hand, these algorithms require global Lipschitz continuity for $\nabla f$, \ie, there exists $L>0$ such that $\|\nabla f(\x)- \nabla f(\y)\|\leq L\|\x-\y\|$ for any $\x,\y\in\R^n$.
This property is sometimes restrictive, and
in some applications, the global Lipschitz continuity does not hold.

Bolte \etal~\cite{Bolte2018-zt} proposed the Bregman proximal gradient algorithm (BPG)\footnote{The Bregman gradient method was first introduced for the case of $g\equiv0$ as the mirror descent~\cite{Nemirovski1983-yw}. The main motivation of developing the method is perhaps to simplify the computation of the projection operator rather than to weaken the assumption of global Lipschitz continuity. For example, the mirror descent is used for minimizing the Kullback--Leibler divergence over the unit simplex~\cite[Example 9.10]{Beck2017-qc}.}, which does not require the global Lipschitz continuity for $\nabla f$:
\begin{align}
    \x^{k+1} = \argmin_{\x\in\cl{C}}\left\{\langle\nabla f(\x^k),\x-\x^k\rangle+g(\x)+\frac{1}{\lambda}D_\phi(\x,\x^k)\right\},\label{subprob:bpg}
\end{align}
where $\lambda > 0$ and $D_\phi(\x,\y) = \phi(\x) - \phi(\y) - \langle\nabla\phi(\y),\x-\y\rangle$ is the Bregman distance~\cite{Bregman1967-rf} with a convex and continuously differentiable function $\phi:\R^n\to(-\infty,+\infty]$.
The method has been refined and extended from various perspectives, as described below.
Ding \etal~\cite{Ding2023-jf} proposed the stochastic BPG and applied it to deep learning problems.
Gao \etal~\cite{Gao2023-ri} proposed an alternating version of BPG and its acceleration version.
Hanzely \etal~\cite{Hanzely2021-sz} proposed an accelerated BPG for convex optimization problems.
Takahashi \etal~\cite{Takahashi2022-ml} proposed a Bregman proximal gradient algorithm exploiting the difference of convex functions (DC) structure and its acceleration version for nonconvex optimization problems.

BPG globally converges to a stationary point (also called a critical point) under the smooth adaptable property (so-called $L$-smad)~\cite{Bolte2018-zt} for $(f,\phi)$ (see Definition~\ref{def:l-smad}), also called relative smoothness~\cite{Lu2018-ii}.
The smooth adaptable property is more permissive than the global Lipschitz continuity.
From the viewpoint of the smooth adaptable property,~\eqref{subprob:bpg} is majorized by the upper approximation of $f$ with $\phi$.
While Subproblem~\eqref{subprob:bpg} is expressed in a closed form when $g$ and $\phi$ have simple structures, such a combination of $g$ and $\phi$ is limited to some examples, e.g.,~\cite{Bauschke2017-hg,Bolte2018-zt,Dragomir2021-rv,Takahashi2023-uh}.

One of the weaknesses of BPG is that Subproblem~\eqref{subprob:bpg} is not easily solved for a general $\phi$.
Even if $g\equiv0$, it is not clear whether the subproblem of BPG has a closed-form solution.
In fact, Subproblem~\eqref{subprob:bpg} reduces to $\min_{\x}\phi(\x) + \langle\lambda\nabla f(\x^k) - \nabla\phi(\x^k),\x\rangle$, which often have no closed-form solutions (for example, $\phi(\x) = \frac{1}{p}\|\x\|_p^p$ with $1 < p\in\R$).
To resolve this issue, assuming that $\phi$ is $\mathcal{C}^2$, we define the approximate Bregman distance $\tilde{D}_\phi$, given by
\begin{align*}
    \tilde{D}_\phi(\x,\y) := \frac{1}{2}\langle\nabla^2\phi(\y)(\x-\y),\x - \y\rangle\simeq D_\phi(\x,\y),
\end{align*}
which is the second-order approximation of $\phi(\x)$ around $\y$ in $D_\phi(\x,\y)$.
Using $\tilde{D}_\phi$ instead of $D_\phi$, we propose a new algorithm, named the approximate Bregman proximal gradient algorithm (ABPG).
Because the subproblem of ABPG can be formulated as a quadratic formula, it is easier to solve than~\eqref{subprob:bpg}.
When $\phi$ is strongly convex and $g\equiv0$, the subproblem of ABPG always has a closed-form solution given by $\x^{k+1} = \x^k - \frac{1}{\lambda}\nabla^2\phi(\x^k)^{-1}\nabla f(\x^k)$.
When $\phi$ is separable, \ie, $\nabla^2\phi(\x^k)$ is diagonal, and $g$ is prox-friendly, the subproblem of ABPG reduces to the proximal calculus (see also Section~\ref{sec:numerical-experiments} and Remark~\ref{remark:lp} for specific examples).
Table~\ref{tab:example-abpg-subprob} summarizes these facts.

\begin{table}[t!]
    \centering
    \caption{ABPG has a closed-form solution in several settings when $t_k = 1$ (see Algorithm~\ref{alg:abpg}).
    Let $f_0$ be a continuously differentiable function with Lipschitz continuous gradients, $\phi_0$ be a continuously differentiable convex function, and $\prox_g$ be a proximal mapping of $g$.
    Let $v_i$ be the $i$th element of $\nabla f(\x^k)$ and $s_i$ be the $(i,i)$ element of $\nabla^2\phi(\x^k)^{-1}$.
    The Kullback--Leibler divergence $D_{\textup{KL}}(\cdot,\cdot)$ will be defined by~\eqref{def:kl-divergence}.
    Let $\A\in\R_+^{m\times n}, \b\in\R_+^m, \W\in\R_+^{m \times r}, \H\in\R_+^{r\times n}$, and $\B\in\R_+^{m\times n}$.
    See also Section~\ref{sec:numerical-experiments}.}
    \begingroup
    \renewcommand{\arraystretch}{1.4}
    \begin{tabular}{c|c|l}
        Reference & $(f, \phi)$ & Calculus of the subproblem\\\hline
        Sects.~\ref{sec:the-lp-regularized-least-squares-problem} \&~\ref{sec:the-lp-regularized-least-squares-problem-with-constraints} & $f = f_0 + \theta\phi_0$, $\phi = \phi_0 + \frac{1}{2}\|\cdot\|^2$ with a separable $\phi_0$ & \multirow{2}{*}{For separable $g(\x) = \sum_{i}g_i(x_i)$,}
        \\ 
        Sect.~\ref{sec:nonnegative-linear-system},~\cite{Bauschke2017-hg}& $f(\x) = D_{\textup{KL}}(\A\x,\b)$, $\phi(\x) = \sum_{i}x_i\log x_i + \frac{1}{2}\|\x\|^2$ & \\ 
        \cite{Bauschke2017-hg}& $f(\x) = D_{\textup{KL}}(\b,\A\x)$, $\phi(\x) = -\sum_{i}\log x_i$ & \multirow{2}{*}{$\x^{k+1} = \left(\prox_{\lambda s_i g_i}(x^k_i - s_iv_i)\right)_{i=1}^n$}\\
        \cite{Takahashi2024-cu}\tablefootnote{ABPG can be applied to the auxiliary function $\hat{f}_k$ defined in~\cite{Takahashi2024-cu}.}& $f(\W,\H) = D_{\textup{KL}}(\B,\W\H)$, $\phi(\x) = -\sum_{i}\log x_i$ & \\ \hline
        Sect.~\ref{sec:lp-loss} & Any $(f, \phi)$ is $L$-smad, $\phi(\x)$ is strongly convex. & $g\equiv0 \implies$ \\ 
        & & $\x^{k+1} = \x^k - \frac{1}{\lambda}\nabla^2\phi(\x^k)^{-1}\nabla f(\x^k)$ \\ \hline
    \end{tabular}
    \endgroup
    \label{tab:example-abpg-subprob}
\end{table}

\begin{table}[!t]
    \centering
    \caption{ABPG corresponds to existing algorithms in setting some $\nabla^2 \phi(\x^k)$.
    In this table, $\kappa_k$ is a positive scalar, and $\Q_k$ is a symmetric positive definite matrix at the $k$th iteration.
    Note that $f$ is required to be convex for some algorithms.}
    \begingroup
    \renewcommand{\arraystretch}{1.4}
    \begin{tabular}{c|c|c}
        $\nabla^2\phi(\x^k)$ & Algorithm & Algorithm ($g\equiv0$) \\\hline
        $\nabla^2\phi(\x^k) = \I$ & Proximal gradient algorithm~\cite{Bruck1975-et,Lions1979-id,Passty1979-sk} & Gradient descent algorithm\\
        $\nabla^2\phi(\x^k) = \nabla^2 f(\x^k)$ & Proximal Newton method~\cite{Patriksson1999-yp} & Newton method\\
        $\nabla^2\phi(\x^k) = \nabla^2 f(\x^k)+\frac{\kappa_k}{2}\I$ & Regularized proximal Newton method~\cite{Yue2019-da} & Regularized Newton method~\cite{Li2004-hi}\\
        $\nabla^2\phi(\x^k) =\Q_k$ & Proximal quasi-Newton method~\cite{Becker2012-ji,Lee2014-zv} & Quasi-Newton method \\ \hline
    \end{tabular}
    \endgroup
    \label{tab:abpg-existing}
\end{table}

Similarly to BPG, we can analyze the convergence results of ABPG under the smooth adaptable property.
ABPG uses the line search procedure to ensure the accuracy of the approximate Bregman distance $\tilde{D}_\phi(\x,\y)$.
We have proven that the line search procedure is well-defined under the smooth adaptable property.
We have established the global subsequential convergence under standard assumptions.
We have also established the global convergence for the special case $g\equiv0$ under the Kurdyka--\L ojasiewicz (KL) property.

ABPG includes some existing algorithms depending on the choice of $\phi$.
We summarize this in Table~\ref{tab:abpg-existing} (also includes the case $g\equiv0$).
When $\phi(\x) = \frac{1}{2}\|\x\|^2$, \ie~$\nabla^2\phi(\x) = \I$, ABPG is equivalent to the proximal gradient method.
Assuming that $f$ is convex and $\mathcal{C}^2$, when $\phi(\x) = f(\x)$, ABPG is equivalent to the proximal Newton method, and when $\phi(\x) = f(\x) + \frac{\kappa_k}{2}\|\x\|^2$, \ie~$\nabla^2\phi(\x^k) = \nabla f(\x^k) + \kappa_k\I$ for $\kappa_k > 0$, ABPG is equivalent to the regularized proximal Newton method~\cite{Yue2019-da} with a parameter $\kappa_k > 0$.
When $\nabla^2\phi(\x^k) = \Q_k$ with symmetric positive definite matrices $\Q_k\in\R^{n \times n}$ at the $k$th iteration, ABPG is equivalent to the proximal quasi-Newton method.
ABPG corresponds to these existing algorithms with the above $\nabla^2\phi$.
The convergence properties of these algorithms are often investigated under the assumption of the global Lipschitz continuity for $\nabla f$, but this paper gives convergence guarantees under the weaker assumption, $L$-smad property.

In numerical experiments, we have applied ABPG to the convex $\ell_p$ regularized least squares problem~\cite{Chung2019-kg,Wen2016-fg}, the $\ell_p$ loss problem~\cite{Maddison2021-nt}, and the nonnegative linear system~\cite{Bauschke2017-hg}.
In the $\ell_p$ regularized least squares problem, $f$ includes the $p$th power of the $\ell_p$ norm term with $p > 1$.
Its gradient for a general $p$ is not globally Lipschitz continuous, and it is also not locally Lipschitz continuous at $\x \in (-1,1)^n$ for $p \in(1,2)$.
Moreover, when $\phi$ includes the $\ell_p$ norm in this case, the subproblem of BPG is the minimization of a $p$th degree polynomial function.
While the subproblem of BPG has no closed-form solutions, that of ABPG has a closed-form solution even if $g \not\equiv0$.
We also have compared ABPG with BPG on the application to the nonnegative linear system.

This paper is organized as follows.
Section~\ref{sec:preliminaries} summarizes key notations such as subdifferentials, the Bregman distance, the smooth adaptable property, and the KL property.
These mathematical tools are necessary to tackle nonconvex optimization problems.
Section~\ref{sec:proposed-algorithm} proposes the approximate Bregman proximal gradient algorithm, which does not require the global Lipschitz continuity for $\nabla f$.
Section~\ref{sec:convergence-analysis} establishes the properties and global convergence of our proposed algorithm.
We proved the global subsequential convergence under standard assumptions and the global convergence when $g\equiv0$.
Section~\ref{sec:numerical-experiments} shows applications to the $\ell_p$ regularized least squares problem, the $\ell_p$ loss problem, and the nonnegative linear system.

\paragraph*{Notation}
In what follows, we use the following notations.
Let $\R, \R_+$, and $\R_{++}$ be the set of real numbers, nonnegative real numbers, and positive real numbers, respectively.
Let $\R^n, \R^n_+$, and $\R_{++}^n$ be the real space of $n$ dimension, the nonnegative orthant of $\R^n$, and the positive orthant of $\R^n$, respectively.
Let $\R^{n \times m}$ be the set of $n \times m$ real matrices.
Vectors and matrices are shown in boldface.
The all one vector is $\bm{1}\in \R^n$, the all zero vector is $\bm{0}\in \R^n$, and the identity matrix is $\bm{I}\in \R^{n\times n}$.
Let $|\bm{x}|$ and $\bm{x}^p$ be the elementwise absolute and $p$th power vectors of $\bm{x}\in\R^n$, respectively.
Given a real number $p\geq 1$, the $\ell_p$ norm is defined by $\|\bm{x}\|_p = (\sum_{i=1}^n|x_i|^p)^{1/p}$.
Let $\lambda_{\max}(\bm{M})$ be the largest eigenvalue of a symmetric matrix $\bm{M}\in\R^{n \times n}$.

Let $\interior{C}$ and $\cl{C}$ be the interior and the closure of a set $C\subset\R^n$, respectively.
We also define the distance from a point $\bm{x}\in\R^n$ to $C$ by $\dist(\bm{x}, C) := \inf_{\bm{y}\in C}\|\bm{x}-\bm{y}\|$.
The indicator function $\delta_C$ is defined by $\delta_C(\x) = 0$ for $\x\in C$ and $\delta_C(\x) = +\infty$ otherwise.
The sign function $\sgn(x)$ is defined by $\sgn(x) = -1$ for $x < 0$,  $\sgn(x) = 0$ for $x = 0$, and $\sgn(x) = 1$ for $x > 0$.

\section{Preliminaries}\label{sec:preliminaries}

\subsection{Subdifferentials}
We first introduce the definitions of subdifferentials.
For an extended-real-valued function $f:\R^n\to\extended$, the effective domain of $f$ is defined by $\dom f := \{\x\in\R^n \mid f(\x) < +\infty \}$.
The function $f$ is said to be proper if $f(\x) > -\infty$ for all $\x\in\R^n$ and $\dom f \neq \emptyset$.
\begin{definition}[{Subdifferentials~\cite[Definition 8.3]{Rockafellar1997-zb}}]
    \label{def:limiting-subdiff}
    Let $f:\R^n\to\extendedleft$ be a proper and lower semicontinuous function.
    \begin{enumerate}
        \item The regular subdifferential of $f$ at $\x\in\dom f$ is defined by
        \begin{align*}
            \hat{\partial} f(\x) = \left\{\bm{\xi}\in\R^n \ \middle|\ \liminf_{\y\to \x,\, \y\neq \x}\frac{f(\y) - f(\x) - \langle \bm{\xi}, \y - \x \rangle}{\|\x - \y\|}\geq 0\right\}.
        \end{align*}
        When $\x\notin\dom\phi$, we set $\hat{\partial}f(\x) = \emptyset$.
        \item The limiting subdifferential of $f$ at $\x\in\R^n$ is defined by
        \begin{align*}
            \partial f(\x) = \left\{\bm{\xi}\in\R^n \ \middle|\ \exists \x^k \xrightarrow{f} \x,\ \bm{\xi}^k\to\bm{\xi}\ \mathrm{such\ that}\ \bm{\xi}^k\in\hat{\partial}f(\x^k)\ \mathrm{for\ all}\ k \right\},
        \end{align*}
        where $\x^k \xrightarrow{f} \x$ means $\x^k \to \x$ and $f(\x^k)\to f(\x)$.
    \end{enumerate}
\end{definition}
In general, $\hat{\partial}f(\x)\subset\partial f(\x)$ holds for any $\x\in\R^n$~\cite[Theorem 8.6]{Rockafellar1997-zb}.
See also~\cite[Example 1.31]{Mordukhovich2018-sk} and~\cite{Rockafellar1997-zb} for examples of subdifferential calculus.
We also define $\dom\partial f := \{\x\in\R^n\mid\partial f(\x)\neq\emptyset\}$.
Note that when $f$ is convex, the limiting subdifferential coincides with the (classical) subdifferential~\cite[Proposition 8.12]{Rockafellar1997-zb}.

\subsection{Bregman Distances}
Let $C$ be a nonempty open convex subset of $\R^n$.
We introduce the kernel generating distance~\cite{Bolte2018-zt} and the Bregman distance.
\begin{definition}[Kernel generating distance~{\cite[Definition 2.1]{Bolte2018-zt}}]
    Associated with $C$, a function $\phi:\R^n\to\extendedleft$ is called a kernel generating distance if it satisfies the following conditions:
    \begin{enumerate}
        \item $\phi$ is proper, lower semicontinuous, and convex, with $\dom\phi\subset\cl{C}$ and $\dom\partial\phi = C$.
        \item $\phi$ is $\mathcal{C}^1$ on $\interior\dom\phi \equiv C$.
    \end{enumerate}
    We denote the class of kernel generating distances associated with $C$ by $\mathcal{G}(C)$.
\end{definition}
\begin{definition}[Bregman distance~\cite{Bregman1967-rf}]
    Given a kernel generating distance $\phi\in\mathcal{G}(C)$, a Bregman distance $D_\phi:\dom\phi\times\interior\dom\phi\to\R_+$ associated with $\phi$ is defined by
    \begin{align}
        D_\phi(\x,\y) = \phi(\x) - \phi(\y) - \langle\nabla\phi(\y),\x-\y\rangle.
        \label{BregmanDist}
    \end{align}
\end{definition}
Note that the Bregman distance $D_\phi$ is not a distance because it does not satisfy the symmetry and the triangle inequality.
Due to the convexity of $\phi$, it holds that $D_\phi(\x,\y)\geq0$ for any $\x\in\dom\phi$ and $\y\in\interior\dom\phi$.
In addition, when $\phi$ is strictly convex, $D_\phi(\x,\y)=0$ holds if and only if $\x = \y$.
\begin{example}
    We show some famous examples of Bregman distances:
    \begin{itemize}
        \item Squared Euclidean distance: Let $\phi(\x) = \frac{1}{2}\|\x\|^2$ and $\dom\phi = \R^n$. 
        Then, we have the squared Euclidean distance $D_\phi(\x,\y)=\frac{1}{2}\|\x - \y\|^2$.
        \item Kullback--Leibler divergence~\textup{\cite{Kullback1951-yv}}: The Boltzmann--Shannon entropy $\phi(\x) = \sum_{i=1}^nx_i\log x_i$ with $0\log 0 = 0$ and $\dom\phi = \R_+^n$.
        Then, we obtain the Kullback--Leibler divergence $D_\phi(\x,\y)=\sum_{i=1}^n x_i\log \frac{x_i}{y_i}$ on the unit simplex $\{\x\in\R_+^n\mid\sum_{i=1}^n x_i = 1\}$.
        \item Itakura--Saito divergence~\textup{\cite{Itakura1968-en}}: The Burg entropy $\phi(\x) = -\sum_{i=1}^n\log x_i$ and $\dom\phi = \R^n_{++}$.
        Then, we obtain the Itakura--Saito divergence $D_\phi(\x,\y)=\sum_{i=1}^n \left(\frac{x_i}{y_i} - \log\frac{x_i}{y_i} - 1\right)$.
    \end{itemize}
\end{example}
See~\cite{Bauschke2017-hg,Bauschke1997-vk,Lu2018-ii} and~\cite[Table 2.1]{Dhillon2008-zz} for more examples.
\subsection{Smooth Adaptable Property}
Now let us define the smooth adaptable property~\cite{Bolte2018-zt}, also called relative smoothness~\cite{Lu2018-ii}.
\begin{definition}[$L$-smooth adaptable property~{\cite[Definition 2.2]{Bolte2018-zt}}]\label{def:l-smad}
    Consider a pair of functions $(f, \phi)$ satisfying the following conditions:
    \begin{enumerate}
        \item $\phi\in\mathcal{G}(C)$,
        \item $f:\R^n\to\extendedleft$ is a proper and lower semicontinuous function with $\dom\phi\subset\dom f$, which is $\mathcal{C}^1$ on $C = \interior\dom\phi$.
    \end{enumerate}
    The pair $(f,\phi)$ is called $L$-smooth adaptable ($L$-smad) on $C$ if there exists $L > 0$ such that $L\phi - f$ and $L\phi + f$ are convex on $C$.
\end{definition}
The smooth adaptable property for the pair $(f,\phi)$ implies that $f$ is globally majorized by $\phi$ with some constant $L > 0$.
The smooth adaptable property induces a useful lemma.
\begin{lemma}[Extended descent lemma~{\cite[Lemma 2.1]{Bolte2018-zt}}]
    \label{lemma:extended-descent}
    The pair of functions $(f, \phi)$ is $L$-smad on $C$ if and only if
    \begin{align*}
        |f(\x) - f(\y) - \langle\nabla f(\y),\x-\y\rangle|\leq LD_\phi(\x,\y), \quad \forall\x,\y\in\interior\dom\phi.
    \end{align*}
\end{lemma}
Recalling $\phi(\x)=\frac{1}{2}\|\x\|^2$ on $C = \R^n$, Lemma~\ref{lemma:extended-descent} coincides with the classical descent lemma, \ie, $\nabla f$ is Lipschitz continuous.

\subsection{Kurdyka--\L ojasiewicz Property}
The Kurdyka--\L ojasiewicz (KL) property plays a central role in global convergence.
Attouch~\etal~\cite{Attouch2010-nr} extended the \L ojasiewicz gradient inequality~\cite{Kurdyka1998-zd,Lojasiewicz1963-sm} to nonsmooth functions.

Given $\eta>0$, let $\Xi_{\eta}$ denote the set of all continuous concave functions $\psi:[0, \eta) \to \R_+$ that are $\mathcal{C}^1$ on $(0, \eta)$ with positive derivatives and which satisfy $\psi(0) = 0$.
Here, we define the Kurdyka--\L ojasiewicz property.
\begin{definition}[Kurdyka--\L ojasiewicz property~{\cite[Definition 7]{Attouch2010-nr}}]
    \label{def:kl}
    Let $f: \R^n \to\extendedleft$ be a proper and lower semicontinuous function.
    \begin{enumerate}
        \item $f$ is said to have the Kurdyka--\L ojasiewicz (KL) property at $\hat{\x}\in\dom\partial f$ if there exist $\eta\in(0, +\infty]$, a neighborhood $U$ of $\hat{\x}$, and a function $\psi\in\Xi_{\eta}$ such that, for all
        \begin{align*}
            \x\in U \cap\{\x\in\R^n \mid f(\hat{\x})<f(\x)<f(\hat{\x})+\eta\},
        \end{align*}
        the following inequality holds:
        \begin{align}
            \label{ineq:kl}
            \psi'(f(\x) - f(\hat{\x})) \dist(\bm{0}, \partial f(\x)) \geq 1.
        \end{align}
        \item If $f$ has the KL property at each point of $\dom\partial f$, then it is called a KL function.
    \end{enumerate}
\end{definition}
Let $\psi(s) = cs^{1-\theta}$ for some $\theta\in[0,1)$ and $c > 0$.
The parameter $\theta$ is called the KL exponent, which affects the rate of convergence.
When $f$ is $\mathcal{C}^1$,~\eqref{ineq:kl} coincides with the \L ojasiewicz gradient inequality~\cite{Kurdyka1998-zd,Lojasiewicz1963-sm}, defined by
\begin{align*}
    |f(\x) - f(\hat{\x})|^{\theta}\leq c(1-\theta)\|\nabla f(\x)\|.
\end{align*}
When $f$ is $\mathcal{C}^1$, the KL exponent $\theta$ is also called the \L ojasiewicz exponent.
Li and Pong~\cite{Li2018-mf} established calculus rules of the KL exponent.

The uniformized KL property is induced from the KL property.
\begin{lemma}[Uniformized KL property~{\cite[Lemma 6]{Bolte2014-td}}]
    \label{lemma:uniformized-kl}
    Suppose that $f:\R^n\to(-\infty, +\infty]$ is a proper and lower semicontinuous function, and let $\Gamma$ be a compact set.
    If $f$ is constant on $\Gamma$ and has the KL property at each point of $\Gamma$, then there exist positive scalars $\epsilon, \eta > 0$, and $\psi\in\Xi_{\eta}$ such that
    \begin{align*}
        \psi'(f(\x) - f(\hat{\x})) \dist(\bm{0}, \partial f(\x))\geq1,
    \end{align*}
    for any $\hat{\x}\in\Gamma$ and any $\x\in\R^n$ satisfying $\dist(\x, \Gamma) < \epsilon$ and $f(\hat{\x}) < f(\x) < f(\hat{\x}) + \eta$.
\end{lemma}

\section{Proposed Algorithm}\label{sec:proposed-algorithm}
Throughout this paper, we make the following assumptions.
\begin{assumption}\
\label{assu:function}
\begin{enumerate}
    \item $\phi\in\mathcal{G}(C)$ with $\cl{C} = \cl{\dom\phi}$ is $\mathcal{C}^2$ on $C = \interior\dom\phi$.
    \item $f:\R^n\to\extendedleft$ is a proper and lower semicontinuous function with $\dom \phi \subset\dom f$, which is $\mathcal{C}^1$ on $C$.
    \item $g:\R^n\to\extendedleft$ is a proper and convex function with $\dom g\cap C \neq\emptyset$.
    \item $\Psi^* := \inf_{\x\in\cl{C}}\Psi(\x) > -\infty$.
    \item For any $\x\in\interior\dom\phi$ and $\lambda > 0$, $\lambda g(\u) + \frac{1}{2}\langle\nabla^2\phi(\x)(\u - \x), \u - \x\rangle$ is supercoercive, that is,
    \begin{align*}
        \lim_{\|\u\|\to\infty}\frac{\lambda g(\u) + \frac{1}{2}\langle\nabla^2\phi(\x)(\u - \x), \u - \x\rangle}{\|\u\|} = \infty.
    \end{align*}
\end{enumerate}
\end{assumption}

Assumption~\ref{assu:function}(i-iv) are standard for Bregman-type algorithms~\cite{Bolte2018-zt,Ding2023-jf,Takahashi2022-ml} and satisfied in practice.
Assumption~\ref{assu:function}(v) is a coercive condition, which is used for the well-posedness of the ABPG mapping (Lemma~\ref{lemma:well-posedness-abpg}).
If $\phi$ is strongly convex, then Assumption~\ref{assu:function}(v) holds~\cite[Corollary 11.17]{Bauschke2017-jp}.
\begin{remark}
    \label{remark:effective-domain}
    In Problem~\eqref{prob:minimization}, $\cl{C}$ is more of the effective domain of $\phi$ than a constraint.
    We use $C = \interior\dom\phi$ to guarantee that points are in the effective domain of $\phi$.
    Therefore, $C$ depends on the choice of $\phi$.
    In Sections~\textup{\ref{subsec:global-subsequential-convergence}} and~\textup{\ref{subsec:global-convergence}}, we use $C\equiv\R^n$ following~\textup{\cite{Bolte2018-zt}}.
\end{remark}

\subsection{Approximate Bregman Proximal Gradient Algorithm}\label{subsec:abpg}
The Bregman proximal gradient mapping~\cite{Bolte2018-zt} at $\x \in C$ is defined by
\begin{align}
    \mathcal{T}_\lambda(\x) :=& \argmin_{\u\in\cl{C}}\left\{\langle\nabla f(\x), \u - \x\rangle + g(\u) + \frac{1}{\lambda}D_\phi(\u,\x)\right\},\label{def:bpg-mappring}
\end{align}
where $\lambda > 0$ is called a stepsize.
For general $\phi$ and $g$, it is difficult to calculate $\mathcal{T}_\lambda(\x)$ in a closed form.
Using Assumption~\ref{assu:function}(i), we obtain the second-order approximation to $\phi(\u)$ around $\x$, given by
\begin{align*}
    D_\phi(\u,\x) &= \phi(\u) - \phi(\x) - \langle\nabla\phi(\x),\u - \x\rangle\\
    &\simeq\phi(\x) + \langle\nabla\phi(\x),\u-\x\rangle+\frac{1}{2}\langle\nabla^2\phi(\x)(\u-\x),\u - \x\rangle\\
    &\quad- \phi(\x) - \langle\nabla\phi(\x),\u - \x\rangle\\
    &=\frac{1}{2}\langle\nabla^2\phi(\x)(\u-\x),\u - \x\rangle =: \tilde{D}_\phi(\u,\x).
\end{align*}
We call $\tilde{D}_\phi(\u,\x)\geq0$ the (second-order) approximate Bregman distance.
Note that it does not always hold that $D_\phi(\u,\x) \leq \tilde{D}_\phi(\u,\x)$ or $D_\phi(\u,\x) \geq \tilde{D}_\phi(\u,\x)$ for any $\u,\x\in\R^n$.
Because of this, our algorithm shown later needs a line search procedure.
We consider, instead of~\eqref{def:bpg-mappring}, the approximate Bregman proximal gradient mapping at $\x \in C$, defined by
\begin{align}
    \tilde{\mathcal{T}}_\lambda(\x) :=& \argmin_{\u\in\cl{C}}\left\{\langle\nabla f(\x), \u - \x\rangle + g(\u) + \frac{1}{\lambda}\tilde{D}_\phi(\u,\x)\right\}.\label{def:abpg-mappring}
\end{align}
Using Assumption~\ref{assu:function}(iii) and the positive semidefiniteness of $\nabla^2\phi(\x)$,~\eqref{def:abpg-mappring} is a convex optimization problem.
In addition, $\tilde{\mathcal{T}}_\lambda(\x)$ is a singleton when $\phi$ is strongly convex.
If $\phi$ is not strongly convex, it is enough to use $\phi_0 := \phi + \frac{\sigma}{2}\|\cdot\|^2$ with $\sigma > 0$ as a new kernel generating distance.
Setting $\phi = \frac{1}{2}\|\cdot\|^2$, \ie, $\nabla^2\phi(\x) = \I$, $\tilde{\mathcal{T}}_\lambda(\x)$ is equivalent to a proximal mapping.
Additionally, we make the following assumption.
\begin{assumption}
    \label{assu:abpg-mapping}
    For any $\x \in C$ and $\lambda > 0$, we have $\tilde{\mathcal{T}}_\lambda(\x) \subset C$.
\end{assumption}
Assumption~\ref{assu:abpg-mapping} guarantees that a sequence generated by ABPG is in $C\equiv\interior\dom\phi$ (see also Remark~\ref{remark:effective-domain}).
Assumption~\ref{assu:abpg-mapping} holds when $C\equiv\R^n$.
In the same discussion as~\cite[p.2136]{Bolte2018-zt}, another approach to warrant Assumption~\ref{assu:abpg-mapping} is to consider extending the prox-boundedness~\cite[Definition 1.23]{Rockafellar1997-zb} to the approximate Bregman proximal gradient mapping and its envelope.
Moreover, if $\phi$ is strongly convex, an envelope function $\inf_{\u\in\cl{C}}\left\{\langle\nabla f(\x), \u - \x\rangle + g(\u) + \frac{1}{\lambda}\tilde{D}_\phi(\u,\x)\right\}$ is bounded from below, which implies the prox-boundedness from~\cite[Exercise 1.24]{Rockafellar1997-zb}.
More discussions for Assumption~\ref{assu:abpg-mapping} have been in~\cite{Bolte2018-zt}. 
We can prove that $\tilde{\mathcal{T}}_\lambda(\x)$ is well-posed, as shown in the following lemma, by using Assumptions~\ref{assu:function} and~\ref{assu:abpg-mapping}.
\begin{lemma}[Well-posedness of $\tilde{\mathcal{T}}_\lambda(\x)$]
    \label{lemma:well-posedness-abpg}
    Suppose that Assumptions~\textup{\ref{assu:function}} and~\textup{\ref{assu:abpg-mapping}} hold.
    Let $\x\in\interior\dom\phi$.
    Then, the mapping $\tilde{\mathcal{T}}_\lambda(\x)$ is a nonempty and compact subset of $C$ for any $\lambda > 0$.
\end{lemma}
\begin{proof}
    Take any $\x\in\interior\dom\phi$ and $\lambda > 0$.
    Let
    \begin{align*}
        \Phi(\u) := \langle\lambda\nabla f(\x), \u\rangle + \lambda g(\u) + \frac{1}{2}\langle\nabla^2\phi(\x)(\u-\x),\u - \x\rangle,
    \end{align*}
    which implies $\tilde{\mathcal{T}}_\lambda(\x) = \argmin_{\u\in\cl{C}}\Phi(\u)$.
    We have
    \begin{align*}
        \Phi(\u) &\geq -|\langle\lambda\nabla f(\x), \u\rangle| + \lambda g(\u) + \frac{1}{2}\langle\nabla^2\phi(\x)(\u-\x),\u - \x\rangle\\
        &= \|\u\|\left(\frac{\lambda g(\u) + \frac{1}{2}\langle\nabla^2\phi(\x)(\u-\x),\u - \x\rangle}{\|\u\|} - \frac{|\langle\lambda\nabla f(\x), \u\rangle|}{\|\u\|}\right).
    \end{align*}
    Because $\lambda g(\u) + \frac{1}{2}\langle\nabla^2\phi(\x)(\u-\x),\u - \x\rangle$ is supercoercive, we have $\lim_{\|\u\|\to\infty}\Phi(\u) = \infty$.
    Since $\Phi$ is also proper and continuous due to the convexity of $g$, we can use Weierstrass's theorem (see, for example,~\cite[Theorem 1.9]{Rockafellar1997-zb}).
    It follows that $\tilde{\mathcal{T}}_\lambda(\u)$ is a nonempty and compact set.
\end{proof}

Inspired by existing Newton-type methods, we take the search direction procedure and the line search procedure.
The search direction of our proposed method solves the subproblem at $\x^k \in C$,
\begin{align}
    \d^k=\argmin_{\x^k + \d\in\cl{C}}\left\{\langle\nabla f(\x^k),\d\rangle + g(\x^k+\d) + \frac{1}{2\lambda}\langle\nabla^2\phi(\x^k)\d,\d\rangle\right\},\label{subprob:direction}
\end{align}
where $\d$ is a variable and $\x^k$ is fixed.
Subproblem~\eqref{subprob:direction} is essentially the same as \eqref{def:abpg-mappring} of $\tilde{\mathcal{T}}_\lambda(\x^k)$ due to $\d^k = \tilde{\mathcal{T}}_\lambda(\x^k) - \x^k$.
When $\phi$ and $g$ are separable,~\eqref{subprob:direction} may be solved in a closed-form expression.
Because $\nabla^2\phi(\x)$ is a diagonal matrix, \eqref{subprob:direction} reduces to the proximal calculus (see, for the proximal calculus,~\cite[Chapter 6 and Appendix B]{Beck2017-qc}).
For example, when $\phi(\x) = \frac{1}{p}\|\x\|_p^p$ and $g(\x) = \theta_1\|\x\|_1$ with $\theta_1 > 0$, we have a closed-form solution of~\eqref{subprob:direction} (see Remark~\ref{remark:lp}).

We use the line search procedure to select $t_k > 0$ at the $k$th iteration so that the sufficient decreasing condition holds for $\x^{k+1} = \x^k + t_k\d^k$,
\begin{align*}
    \Psi(\x^{k+1}) \leq \Psi(\x^k)+\alpha t_k(\langle\nabla f(\x^k),\d^k\rangle + g(\x^k+\d^k) - g(\x^k)),
\end{align*}
where $\x^k \in C$, $\alpha\in(0,1)$, and $\d^k$ is given by~\eqref{subprob:direction}.
Since $C$ is convex, $\x^k+t_k\d^k = t_k(\x^k + \d^k) + (1-t_k)\x^k \in C$ with $\x^k \in C$ and $\x^k+\d^k \in \tilde{\mathcal{T}}_\lambda(\x^k) \subset C$ (Lemma~\ref{lemma:well-posedness-abpg}).
The simple way to calculate $t_k$ would be the backtracking line search (see, for example,~\cite{Nocedal2006-sb}), and
we also use the backtracking approach.

Now we are ready to describe our algorithm for solving the nonconvex composite optimization problem~\eqref{prob:minimization}.

\begin{algorithm}[H]
    \caption{Approximate Bregman proximal gradient algorithm (ABPG)}
    \label{alg:abpg}
    \begin{algorithmic}[t]
    \Input{A $\mathcal{C}^2$ function $\phi\in\mathcal{G}(C)$ such that the pair $(f, \phi)$ is $L$-smad on $C$.}
    \Initialization{$\x^0\in\interior\dom\phi,\eta\in(0,1),\alpha\in(0,1)$, and $\lambda>0$.}
    \For{$k = 0, 1, 2, \ldots,$}
        \State Compute a search direction with a fixed $\x^k$:
        \begin{align}
            \d^k = \argmin_{\x^k + \d\in\cl{C}}\left\{\langle\nabla f(\x^k),\d\rangle + g(\x^k+\d) + \frac{1}{2\lambda}\langle\nabla^2\phi(\x^k)\d,\d\rangle\right\}.\label{subprob:abpg}
        \end{align}
        \State Set $t_k = 1$.
        \While{$\Psi(\x^k+t_k\d^k) > \Psi(\x^k)+\alpha t_k(\langle\nabla f(\x^k),\d^k\rangle + g(\x^k+\d^k) - g(\x^k))$}
            \State $t_k \leftarrow \eta t_k$
        \EndWhile
        \State Update $\x^{k+1} = \x^k + t_k \d^k$.
    \EndFor
    \end{algorithmic}
\end{algorithm}

\begin{remark}
    Assuming that $f$ is $\mathcal{C}^2$, we see the useful property~\textup{\cite[Proposition 1]{Bauschke2017-hg}}:
    There exists $L > 0$ such that $L\phi - f$ is convex if and only if
    \begin{align}
        \exists L > 0, L\nabla^2\phi(\x)\succeq\nabla^2 f(\x),\quad \forall\x\in\interior\dom\phi.\label{ineq:l-smad-second-order}
    \end{align}
    Condition~\eqref{ineq:l-smad-second-order} is useful for estimating $L$ in applications such that the Poisson linear inverse problem~\textup{\cite[Lemma 7]{Bauschke2017-hg}}, phase retrieval~\textup{\cite[Lemma 5.1]{Bolte2018-zt}\cite[Propositions 5 and 6]{Takahashi2022-ml}}, and blind deconvolution~\textup{\cite[Theorem 1]{Takahashi2023-uh}}.
    Regarding a role of $\frac{1}{\lambda}\nabla^2\phi(\x)$ in~\eqref{def:abpg-mappring},
    Condition \eqref{ineq:l-smad-second-order} implies that $\frac{1}{\lambda}\nabla^2\phi(\x)$ is an upper approximation to $\nabla^2 f(\x)$ when $\lambda \in(0,1/L)$.
\end{remark}

In the next section, we show the well-definedness of the search direction and the line search, and then establish a global convergence result to a stationary point.

\section{Convergence Analysis}\label{sec:convergence-analysis}
Throughout this section, we make the following assumption.
\begin{assumption}
    \label{assu:l-smad}
    The pair $(f, \phi)$ is $L$-smad on $C$.
\end{assumption}
\subsection{Properties of the Proposed Algorithm}
We first show the search direction property.
It induces the decreasing property.
\begin{proposition}[Search direction properties]\label{prop:search-direction}
Suppose that Assumptions~\textup{\ref{assu:function}},~\textup{\ref{assu:abpg-mapping}} and~\textup{\ref{assu:l-smad}} hold.
For any $\lambda > 0$, $\x\in\interior\dom\phi$, and $\d\in\R^n$ defined by 
\begin{align}
    \d = \argmin_{\u}\left\{\langle\nabla f(\x),\u\rangle + g(\x+\u) + \frac{1}{2\lambda}\langle\nabla^2\phi(\x)\u,\u\rangle\right\}
    ~\mbox{s.t. }~ \x + \u\in \cl{C},\label{subprob:direction-prop}
\end{align}
we have
\begin{align}
    &\Psi(\x^+) - \Psi(\x) \leq t(\langle\nabla f(\x), \d\rangle + g(\x+\d) - g(\x)) + o(t\|\d\|),\label{ineq:property-d-taylor}
\end{align}
where $\x^+ = \x + t\d$ and $t\in(0,1]$.
Furthermore, we have
\begin{align}
    &\langle\nabla f(\x),\d\rangle+g(\x+\d)-g(\x)\leq -\frac{1}{2\lambda}\langle\nabla^2\phi(\x)\d,\d\rangle.\label{ineq:property-d-2}
\end{align}
\end{proposition}
\begin{proof}
    By using Taylor's theorem and $g(\x+t\d) \leq t g(\x+\d) + (1-t)g(\x)$ due to the convexity of $g$, we obtain
    \begin{align*}
        \Psi(\x^+) - \Psi(\x) &= f(\x^+) - f(\x) + g(\x^+) - g(\x)\\
        &\leq t(\langle\nabla f(\x), \d\rangle + g(\x+\d) - g(\x)) + o(t\|\d\|),
    \end{align*}
    which proves~\eqref{ineq:property-d-taylor}.

    Using the optimality condition of~\eqref{subprob:direction-prop}, we have
    \begin{align*}
        \langle\nabla f(\x),\d\rangle + g(\x+\d) + \frac{1}{2\lambda}\langle\nabla^2\phi(\x)\d,\d\rangle
        &\leq t\langle\nabla f(\x),\d\rangle + g(\x+t\d) + \frac{t^2}{2\lambda}\langle\nabla^2\phi(\x)\d,\d\rangle\\
        &\leq t\langle\nabla f(\x),\d\rangle + tg(\x+\d) + (1-t)g(\x)\\
        &\quad+ \frac{t^2}{2\lambda}\langle\nabla^2\phi(\x)\d,\d\rangle,
    \end{align*}
    where in the first inequality, function values are compared between the optimal solution $\d$ and a feasible solution $t \d$ with $t\in (0,1]$, and the second inequality holds due to the convexity of $g$.
    Moving the right-hand side to the left-hand side and simplifying the above inequality, we obtain
    \begin{align*}
        (1-t)&\langle\nabla f(\x),\d\rangle+\frac{1-t^2}{2\lambda}\langle\nabla^2\phi(\x)\d,\d\rangle + (1-t)(g(\x+\d) - g(\x))\leq 0,
   \end{align*}        
    and moreover, by using $1-t\geq0$, 
   \begin{align*}
        &\langle\nabla f(\x),\d\rangle + g(\x+\d) - g(\x)\leq -\frac{1+t}{2\lambda}\langle\nabla^2\phi(\x)\d,\d\rangle\leq-\frac{1}{2\lambda}\langle\nabla^2\phi(\x)\d,\d\rangle
    \end{align*}
    holds for any $t \in (0,1].$ 
    The desired inequality is proved.
\end{proof}
Note that Proposition~\ref{prop:search-direction} does not require the $L$-smad property.
Proposition~\ref{prop:search-direction} implies the search direction is the descent direction because, substituting~\eqref{ineq:property-d-2} into~\eqref{ineq:property-d-taylor}, it holds that
\begin{align*}
    \Psi(\x^+) - \Psi(\x) \leq -\frac{t}{2\lambda}\langle\nabla^2\phi(\x)\d,\d\rangle + o(t\|\d\|).
\end{align*}

Let $\alpha\in(0,1)$. 
We next show that our line search procedure is well-defined, \ie, the sufficient decreasing condition
\begin{align}
    \Psi(\x^+) \leq \Psi(\x)+\alpha t(\langle\nabla f(\x),\d\rangle + g(\x+\d) - g(\x))\label{ineq:sufficient-decreasing-condition}
\end{align}
is satisfied after finite steps of the line search procedures under the following assumption.
\begin{assumption}
\label{assu:strongly-convex}
    $\phi$ is $\sigma$-strongly convex on $C$ with $\sigma > 0$.
\end{assumption}
We discussed the strong convexity of $\phi$ in Section~\ref{subsec:abpg}.
\begin{lemma}
    \label{lemma:well-defined-line-search}
    Suppose that Assumptions~\textup{\ref{assu:function}},~\textup{\ref{assu:abpg-mapping}},~\textup{\ref{assu:l-smad}}, and~\textup{\ref{assu:strongly-convex}} hold.
    For any $\alpha\in(0,1)$, $\lambda > 0$, and some $\beta \geq 0$, if
    \begin{align}
        t \leq \min\left\{1,\frac{1-\alpha}{\beta\lambda L}\right\},\label{ineq:step-t}
    \end{align}
    then the sufficient decreasing condition~\eqref{ineq:sufficient-decreasing-condition} is satisfied after finite steps of the line search procedures.
\end{lemma}
\begin{proof}
    When $\d=\zero$,~\eqref{ineq:sufficient-decreasing-condition} obviously holds.
    We will prove~\eqref{ineq:sufficient-decreasing-condition} when $\d\neq\zero$.
    Because the pair $(f,\phi)$ is $L$-smad, using Lemma~\ref{lemma:extended-descent}, we have
    \begin{align}
        \Psi(\x^+) - \Psi(\x) &= f(\x^+) - f(\x) + g(\x^+) - g(\x)\nonumber\\
        &\leq t\langle\nabla f(\x),\d\rangle + LD_\phi(\x^+,\x) + t(g(\x + \d) - g(\x))\nonumber\\
        &=t(\langle\nabla f(\x),\d\rangle + g(\x + \d) - g(\x)) + LD_\phi(\x^+,\x),\label{ineq:first-order}
    \end{align}
    where the inequality holds due to the convexity of $g$.
    Applying Taylor's Theorem to $\phi(\x^+)$ in $D_\phi(\x^+,\x)$ of \eqref{BregmanDist}, there exists $\tau\in(0,1)$ such that
    \begin{align*}
        D_\phi(\x^+,\x) &= \phi(\x) + t\langle\nabla\phi(\x),\d\rangle +\frac{t^2}{2}\langle\nabla^2\phi(\x + \tau t \d)\d,\d\rangle - \phi(\x) - t\langle\nabla\phi(\x),\d\rangle\\
        &= \frac{t^2}{2}\langle\nabla^2\phi(\x + \tau t \d)\d,\d\rangle.
    \end{align*}
    We consider $h(s) = \langle\nabla^2\phi(\x + s\d)\d,\d\rangle > 0$ for $\d\neq\zero$ and $s\in[0,1]$, which is positive and finite-valued because of the strong convexity of $\phi$ and $\x + s\d\in\interior\dom\phi$ for any $s\in[0,1]$.
    We see $h(0) = \langle\nabla^2\phi(\x)\d,\d\rangle > 0$ for $\d\neq\zero$.
    Since $\phi$ is $\mathcal{C}^2$, $h$ is continuous on $[0,1]$.
    Using the extreme value theorem, $h$ has the maximum value on $[0,1]$.
    Because $\nabla^2\phi(\x)$ is positive definite, there exists $\beta \geq 0$ such that
    \begin{align*}
        h(s)\leq\beta\langle\nabla^2\phi(\x)\d,\d\rangle, \quad\forall s\in[0,1],
    \end{align*}
    which implies
    \begin{align}
        D_\phi(\x^+,\x) \leq \frac{\beta t^2}{2}\langle\nabla^2\phi(\x)\d,\d\rangle.\label{ineq:bregman-quadratic-bounded}
    \end{align}
    Note that $\beta$ is finite because $h(s)$ is also finite-valued.
    Substituting~\eqref{ineq:bregman-quadratic-bounded}
    into~\eqref{ineq:first-order}, we obtain
    \begin{align*}
        \Psi(\x^+) - \Psi(\x) &\leq t\left(\langle\nabla f(\x),\d\rangle + g(\x + \d) - g(\x) + \frac{\beta t L}{2}\langle\nabla^2\phi(\x)\d,\d\rangle\right)\\
        &\leq t\left(\langle\nabla f(\x),\d\rangle + g(\x + \d) - g(\x) + \frac{1-\alpha}{2\lambda}\langle\nabla^2\phi(\x)\d,\d\rangle\right)\\
        &\leq t\left(\rho- (1-\alpha)\rho\right)
        =\alpha t\rho,
    \end{align*}
    where the second inequality holds from $t \leq \frac{1-\alpha}{\beta\lambda L}$, and the third inequality follows from the notation $\rho := \langle\nabla f(\x),\d\rangle + g(\x + \d) - g(\x)$ together with using~\eqref{ineq:property-d-2} to the second term. 
    The desired inequality is proved.
\end{proof}

Lemma~\ref{lemma:well-defined-line-search} induces that~\eqref{ineq:sufficient-decreasing-condition} holds after finitely many numbers of backtracking steps.
Substituting~\eqref{ineq:property-d-2} into~\eqref{ineq:sufficient-decreasing-condition}, we obtain
\begin{align}
    \Psi(\x^+) \leq \Psi(\x)-\frac{\alpha t}{2\lambda}\langle\nabla^2\phi(\x)\d,\d\rangle,\label{ineq:sufficient-decrease-hess}
\end{align}
which also implies a sufficient decrease in the objective function value $\Psi$.
Inequality~\eqref{ineq:sufficient-decrease-hess} does not require $0 < \lambda L < 1$ while BPG requires it to prove the sufficient decreasing~\cite[Lemma 4.1]{Bolte2018-zt}.
We can choose $\lambda$ larger than $1/L$, that is, $\lambda L \geq 1$, while $t$ might be small in practice as \eqref{ineq:step-t} implies.

\subsection{Global Subsequential Convergence}\label{subsec:global-subsequential-convergence}
Throughout Sections~\ref{subsec:global-subsequential-convergence} and~\ref{subsec:global-convergence}, we consider~\eqref{prob:minimization} with $C\equiv\R^n$.
In this subsection, we establish the global subsequential convergence, \ie, any accumulation point of the sequence $\{\x^k\}_{k=0}^\infty$ generated by ABPG is a stationary point.
Inspired by Fermat's rule~\cite[Theorem 10.1]{Rockafellar1997-zb}, we use the limiting subdifferential and define the stationary point, also called the critical point~\cite{Bolte2018-zt}.
\begin{definition}
    We say that $\tilde{\x}\in\R^n$ is a stationary point of~\eqref{prob:minimization} if
    \begin{align*}
        \zero\in\partial\Psi(\tilde{\x}).
    \end{align*}
    The set of all stationary points is denoted by $\mathcal{X}$.
    \label{def:set-stationary}
\end{definition}
Using~\cite[Theorem 1.107(ii)]{Mordukhovich2006-xd}, for any $\x\in\R^n$, we obtain
\begin{align*}
    \partial\Psi(\x) \equiv \nabla f(\x) + \partial g(\x),
\end{align*}
where $\partial g(\x)$ is a (classical) subdifferential since $g$ is convex~\cite[Proposition 8.12]{Rockafellar1997-zb}.

To prove the global convergence, we additionally make the following assumptions.
\begin{assumption}\
\label{assu:locally-lip-level-bounded}
\begin{enumerate}
    \item $\nabla f$ and $\nabla \phi$ are Lipschitz continuous on any bounded subset of $\R^n$.
    \item The objective function $\Psi$ is level-bounded, \ie, for any $r\in\R$, lower level sets $\{\x\in\R^n\mid\Psi(\x)\leq r\}$ are bounded.
\end{enumerate}
\end{assumption}
Assumption~\ref{assu:locally-lip-level-bounded}(i) is weaker than the global Lipschitz continuity for $\nabla f$.
Some researchers~\cite{Bolte2018-zt,Gao2023-ri} make Assumption~\ref{assu:locally-lip-level-bounded}(i).
Note that $\nabla\phi$ is locally Lipschitz continuous because $\phi$ is $\mathcal{C}^2$.

Now, we establish the global subsequential convergence.
\begin{theorem}[Global subsequential convergence of ABPG]
    \label{theorem:global-subsequential-convergence}
    Suppose that Assumptions~\textup{\ref{assu:function}}, \textup{\ref{assu:abpg-mapping}}, \textup{\ref{assu:l-smad}}, \textup{\ref{assu:strongly-convex}}, and~\textup{\ref{assu:locally-lip-level-bounded}} hold.
    Let $\{\x^k\}_{k=0}^\infty$ be a sequence generated by ABPG for solving~\eqref{prob:minimization}.
    Then, the following statements hold:
    \begin{enumerate}
        \item The sequence $\{\x^k\}_{k=0}^\infty$ is bounded.
        \item $\lim_{k\to\infty}\|\d^k\|=0$.
        \item Any accumulation point of $\{\x^k\}_{k=0}^\infty$ is a stationary point of~\eqref{prob:minimization}.
    \end{enumerate}
\end{theorem}
\begin{proof}
    (i) Using~\eqref{ineq:sufficient-decrease-hess} and the positive definiteness of $\nabla^2\phi(\x^k)$ due to Assumption~\ref{assu:strongly-convex}, we know $\Psi(\x^{k+1})\leq\Psi(\x^k)$, which induces $\Psi(\x^k)\leq\Psi(\x^0)$ for all $k\in\mathbb{N}$.
    Because of this and Assumption~\ref{assu:locally-lip-level-bounded}(ii), we have the desired result.

    (ii) Again using~\eqref{ineq:sufficient-decrease-hess}, we obtain
    \begin{align}
        \Psi(\x^{k-1}) - \Psi(\x^k) \geq \frac{\alpha t_{k-1}}{\lambda}\langle\nabla^2\phi(\x^{k-1})\d^{k-1},\d^{k-1}\rangle\geq\frac{\alpha\sigma t_{k-1}}{2\lambda}\|\d^{k-1}\|^2,\label{ineq:sufficient-decreasing-euc}
    \end{align}
    where the last inequality holds since $\phi$ is $\sigma$-strongly convex.
    Summing the above inequality from $k = 1$ to $\infty$, we have
    \begin{align*}
        \sum_{k=1}^\infty\frac{\alpha\sigma t_{k-1}}{2\lambda}\|\d^{k-1}\|^2\leq\Psi(\x^0) - \liminf_{N\to\infty}\Psi(\x^N)\leq\Psi(\x^0) - \Psi^* < \infty,
    \end{align*}
    which shows that $\lim_{k\to\infty}\|\d^k\|=0$ because $t_k$ is bounded away from 0 (if necessary, assume that there exists $t_{\min} > 0$ such that $t_{\min} \leq t_k$ for all $k$).

    (iii) Let $\tilde{\x}\in\R^n$ be an accumulation point of $\{\x^k\}_{k=0}^\infty$ and let $\{\x^{k_j}\}_{j=0}^\infty$ be a subsequence such that $\lim_{j\to\infty}\x^{k_j}=\tilde{\x}$ by Bolzano--Weierstrass theorem.
    We have $\d^k\to\zero$ because of (ii) and let $\{\d^{k_j}\}_{j=0}^\infty$ be a subsequence such that $\lim_{j\to\infty}\d^{k_j}=\zero$.
    We substitute $\x^{k_j}$ into the first-order optimality condition of~\eqref{subprob:abpg} to have
    \begin{align*}
        \zero\in\nabla f(\x^{k_j}) + \partial g(\x^{k_j} + \d^{k_j}) + \frac{1}{\lambda}\nabla^2\phi(\x^{k_j})\d^{k_j},
    \end{align*}
    which imply 
    \begin{align}
        -\frac{1}{\lambda}\nabla^2\phi(\x^{k_j})\d^{k_j}\in\nabla f(\x^{k_j}) + \partial g(\x^{k_j} + \d^{k_j}).\label{cond:first-order}
    \end{align}
    Since $\nabla\phi$ is Lipschitz continuous on bounded subsets of $\R^n$ and $\{\x^{k_j}\}_{j=0}^\infty$ is bounded, there exists $A > 0$ such that
    \begin{align*}
        \left\|\frac{1}{\lambda}\nabla^2\phi(\x^{k_j})\d^{k_j}\right\|\leq\frac{A}{\lambda}\|\d^{k_j}\|\to0.
    \end{align*}
    Taking the limit of~\eqref{cond:first-order} and using the lower semicontinuity of $f$ and $g$, we obtain
    \begin{align*}
        \zero\in\nabla f(\tilde{\x}) + \partial g(\tilde{\x}),
    \end{align*}
    which shows that $\tilde{\x}$ is a stationary point of~\eqref{prob:minimization}.
\end{proof}

Before discussing the global convergence, we consider the objective function value $\Psi$ at a stationary point.
Consequently, we show that $\Psi$ is constant on the set of accumulation points of $\{\x^{k_j}\}_{j=0}^\infty$.
\begin{proposition}
    \label{prop:subsequential-objective}
    Suppose that Assumptions~\textup{\ref{assu:function}},~\textup{\ref{assu:abpg-mapping}},~\textup{\ref{assu:l-smad}},~\textup{\ref{assu:strongly-convex}}, and~\textup{\ref{assu:locally-lip-level-bounded}} hold.
    Let $\{\x^k\}_{k=0}^\infty$ be a sequence generated by ABPG for solving~\eqref{prob:minimization}.
    Then, the following statements hold:
    \begin{enumerate}
        \item $\zeta := \lim_{k\to\infty}\Psi(\x^k)$ exists.
        \item $\Psi\equiv\zeta$ on the set of accumulation points of $\{\x^{k_j}\}_{j=0}^\infty$, denoted by $\Omega$.
    \end{enumerate}
\end{proposition}
\begin{proof}
    (i) The sequence $\{\Psi(\x^k)\}_{k=0}^\infty$ is bounded and non-increasing due to Assumption~\ref{assu:function}(iv) and $\Psi(\x^{k+1})\leq\Psi(\x^k)$ (see the proof of Theorem~\ref{theorem:global-subsequential-convergence}).
    Therefore, $\zeta := \lim_{k\to\infty}\Psi(\x^k)$ exists.

    (ii) Let $\tilde{\x}\in\Omega$ be any accumulation point of $\{\x^k\}_{k=0}^\infty$ and let $\{\x^{k_j}\}_{j=0}^\infty$ be a subsequence such that $\lim_{j\to\infty}\x^{k_j}=\tilde{\x}$.
    Note that $\tilde{\x}$ is a stationary point from Theorem~\ref{theorem:global-subsequential-convergence}(iii).
    From Assumption~\ref{assu:function}(ii) and (iii), $f$ and $g$ are continuous on $\interior\dom\phi = \R^n$.
    Therefore, we immediately obtain $\lim_{j\to\infty}\Psi(\x^{k_j}) = \Psi(\tilde{\x}^*) = \zeta$.
    Since $\tilde{\x}$ is arbitrary, $\Psi\equiv\zeta$ on $\Omega$.
\end{proof}

\subsection{Global Convergence: Special Case \texorpdfstring{$g\equiv0$}{g=0}}\label{subsec:global-convergence}
In the previous subsection, we have shown that any accumulation point of a sequence generated by ABPG is a stationary point of~\eqref{prob:minimization}.
We next establish that the sequence converges to a stationary point of~\eqref{prob:minimization}.
From the first-order optimality condition of~\eqref{subprob:abpg}, it follows that
\begin{align}
    \zero\in\nabla f(\x^{k-1}) + \partial g(\x^{k-1}+\d^{k-1})+\frac{1}{\lambda}\nabla^2\phi(\x^{k-1})\d^{k-1}.\label{cond:first-order-dk}
\end{align}
In this subsection, we focus on the special case $g\equiv0$ because the current problem setting that $g$ may not be $\mathcal{C}^1$ makes it difficult to establish the global convergence using \eqref{cond:first-order-dk}.
If $\nabla^2\phi(\bm{x}^k)$ is nonsingular and $g\equiv0$, Subproblem~\eqref{subprob:abpg} is always solved in a closed-form expression, given by
\begin{align*}
    \d^k = -\lambda\nabla^2\phi(\bm{x}^k)^{-1}\nabla f(\x^k).
\end{align*}
In this case, ABPG is an approximation of the mirror descent~\cite{Nemirovski1983-yw}.

\begin{theorem}[Global convergence of ABPG]
    Suppose that Assumptions~\textup{\ref{assu:function}}, \textup{\ref{assu:abpg-mapping}},~\textup{\ref{assu:l-smad}}, \textup{\ref{assu:strongly-convex}}, and \textup{\ref{assu:locally-lip-level-bounded}} hold and that $\Psi$ is a KL function.
    Let $\{\x^k\}_{k=0}^\infty$ be a sequence generated by ABPG for solving~\eqref{prob:minimization}.
    Then, the following statements hold when $g \equiv 0$:
    \begin{enumerate}
        \item $\lim_{k\to\infty}\dist(\zero,\partial\Psi(\x^k))=0$.
        \item The sequence $\{\x^k\}_{k=0}^\infty$ converges to a stationary point of~\eqref{prob:minimization}; moreover, $\sum_{k=1}^\infty\|\x^k - \x^{k-1}\| < \infty$.
    \end{enumerate}
\end{theorem}
\begin{proof}
    (i) Since $\{\x^k\}_{k=0}^\infty$ is bounded and $\Omega$ is the set of accumulation points of $\{\x^k\}_{k=0}^\infty$, it holds that
    \begin{align*}
        \lim_{k\to\infty}\dist(\x^k,\Omega)=0.
    \end{align*}
    We also have $\Omega\subset\mathcal{X}$ by Theorem~\ref{theorem:global-subsequential-convergence}(iii) and Definition~\ref{def:set-stationary}.
    Thus, for any $\varepsilon> 0$, there exists a positive integer $k_0 > 0$ such that $\dist(\x^k,\Omega) < \varepsilon$ for all $k\geq k_0$.

    The subdifferential of $\Psi$ at $\x^k$ for any $k \geq k_0$ is $\partial\Psi(\x^k) = \{\nabla f(\x^k)\}$ due to $g\equiv0$.
    Using Assumption~\ref{assu:locally-lip-level-bounded}(i) and~\eqref{cond:first-order-dk} with $g\equiv0$, there exists $A_0 > 0$ such that
    \begin{align*}
        \|\nabla f(\x^k)\| &= \|\nabla f(\x^k) - \nabla f(\x^{k-1}) + \nabla f(\x^{k-1})\|\\
        &\leq \|\nabla f(\x^k) - \nabla f(\x^{k-1})\| + \|\nabla f(\x^{k-1})\|\\
        &= \|\nabla f(\x^k) - \nabla f(\x^{k-1})\| + \frac{1}{\lambda}\|\nabla^2\phi(\x^{k-1})\d^{k-1}\|\\
        &\leq A_0\|\d^{k-1}\|,
    \end{align*}
    which is equivalent to
    \begin{align}
        \dist(\zero,\partial\Psi(\x^k))\leq A_0\|\d^{k-1}\|,\label{ineq:dist-d}
    \end{align}
    where $k \geq k_0 + 1$.
    Using Theorem~\ref{theorem:global-subsequential-convergence}(ii), we prove the desired result.

    (ii) If there exists an integer $\tilde{k} \geq 0$ for which $\Psi(\x^{\tilde{k}}) = \Psi(\tilde{\x}) = \zeta$, then~\eqref{ineq:sufficient-decreasing-euc} implies $\x^{\tilde{k} + 1} = \x^{\tilde{k}}$.
    This means that $\{\x^k\}_{k=0}^\infty$ finitely converges to a stationary point.

    We next consider the case in which $\Psi(\x^k) > \zeta$ for all $k \geq 0$.
    Since Proposition~\ref{prop:subsequential-objective}(ii) and $\{\Psi(\x^k)\}_{k=0}^\infty$ is non-increasing, it holds that $\Psi(\tilde{\x}) < \Psi(\x^k)$.
    Again using Proposition~\ref{prop:subsequential-objective}(ii), for any $\eta > 0$, there exists a nonnegative integer $k_1 \geq 0$ such that $\Psi(\x^k) < \Psi(\tilde{\x}) + \eta$ for all $ k > k_1$.
    Using Lemma~\ref{lemma:uniformized-kl}, for any $k > l:=\max\{k_0,k_1\}$, we have
    \begin{align}
        \psi'(\Psi(\x^k) - \Psi(\tilde{\x}))\dist(\zero,\partial\Psi(\x^k))\geq1.\label{ineq:uniform-kl-psi}
    \end{align}
    Since $\psi$ is a concave function, it follows that, for any $k \geq l$,
    \begin{align*}
        (\psi(\Psi_k) - \psi(\Psi_{k+1})\dist(\zero,\partial\Psi(\x^k))
        &\geq\psi'(\Psi_k)\dist(\zero,\partial\Psi(\x^k))\left(\Psi(\x^k)-\Psi(\x^{k+1})\right)\\
        &\geq\Psi(\x^k)-\Psi(\x^{k+1})\\
        &\geq A_1\|\d^k\|^2,
    \end{align*}
    where $\Psi_k := \Psi(\x^k) - \zeta$, the second inequality holds due to~\eqref{ineq:uniform-kl-psi}, and the last inequality holds due to~\eqref{ineq:sufficient-decreasing-euc}.
    Substituting~\eqref{ineq:dist-d} into the above inequality, we have
    \begin{align}
        \|\d^k\|^2\leq\frac{A_0}{A_1}\left(\psi(\Psi_k) - \psi(\Psi_{k+1})\right)\|\d^{k-1}\|.\label{ineq:d-sq-obj}
    \end{align}
    Taking the square root of~\eqref{ineq:d-sq-obj} and using the inequality of arithmetic and geometric means, it holds that
    \begin{align*}
        \|\d^k\|&\leq\sqrt{\frac{A_0}{A_1}\left(\psi(\Psi_k) - \psi(\Psi_{k+1})\right)}\sqrt{\|\d^{k-1}\|}\\
        &\leq\frac{A_0}{2A_1}\left(\psi(\Psi_k) - \psi(\Psi_{k+1})\right) + \frac{1}{2}\|\d^{k-1}\|,
    \end{align*}
    that is,
    \begin{align}
        \frac{1}{2}\|\d^k\|\leq\frac{A_0}{2A_1}\left(\psi(\Psi_k) - \psi(\Psi_{k+1})\right) + \frac{1}{2}\|\d^{k-1}\| - \frac{1}{2}\|\d^k\|.\label{ineq:d-obj-bounded}
    \end{align}
    Summing~\eqref{ineq:d-obj-bounded} from $k = l$ to $\infty$, we obtain
    \begin{align*}
        \sum_{k=l}^\infty\|\d^k\| \leq \frac{A_0}{2A_1}
        \psi(\Psi_l)
        + \frac{1}{2}\|\d^{l-1}\| < \infty.
    \end{align*}
    We have $\|\x^{k+1} - \x^k\| = t_k\|\d^k\| \leq \|\d^k\|$ due to $t_k\in(0,1]$ and then know
    \begin{align*}
        \sum_{k=l}^\infty\|\x^{k+1} - \x^k\| \leq\sum_{k=l}^\infty\|\d^k\| < \infty,
    \end{align*}
    which implies that $\sum_{k=1}^\infty\|\x^{k+1} - \x^k\|< \infty$, \ie, the sequence $\{\x^k\}_{k=0}^\infty$ is a Cauchy sequence.
    Consequently, together with Theorem~\ref{theorem:global-subsequential-convergence}(iii), the sequence $\{\x^k\}_{k=0}^\infty$ converges to a stationary point of~\eqref{prob:minimization}.
\end{proof}

Finally, we show the rate of convergence result.
Its proof is almost the same as the existing one (see, for example,~\cite[Theorem 2]{Attouch2009-wf},~\cite{Attouch2010-nr}, and~\cite[Theorem 4]{Takahashi2022-ml}).
\begin{theorem}[Rate of convergence]
    Suppose that Assumptions~\textup{\ref{assu:function}},~\textup{\ref{assu:abpg-mapping}},~\textup{\ref{assu:l-smad}},~\textup{\ref{assu:strongly-convex}}, and~\textup{\ref{assu:locally-lip-level-bounded}} hold.
    Let $\{\x^k\}_{k=0}^\infty$ be a sequence generated by ABPG for solving~\eqref{prob:minimization} and let $\tilde{\x}\in\mathcal{X}$ be a stationary point of~\eqref{prob:minimization}.
    Suppose further that $\Psi$ is a KL function with $\psi$ in the KL inequality~\eqref{ineq:kl} taking form $\psi(s) = cs^{1-\theta}$ for some $\theta\in[0,1)$ and $c > 0$.
    Then, the following statements hold when $g \equiv 0$:
    \begin{enumerate}
        \item If $\theta = 0$, then the sequence $\{\x^k\}_{k=0}^\infty$ converges to $\tilde{\x}$ in a finite number of steps;
        \item If $\theta\in(0,1/2]$, then there exist $c_1 > 0$ and $\eta\in[0,1)$ such that
        \begin{align*}
            \|\x^k - \tilde{\x}\| < c_1\eta^k;
        \end{align*}
        \item If $\theta\in(1/2, 1)$, then there exists $c_2 > 0$ such that 
        \begin{align*}
            \|\x^k - \tilde{\x}\| < c_2k^{-\frac{1-\theta}{2\theta - 1}}.
        \end{align*}
    \end{enumerate}
\end{theorem}

\section{Numerical Experiments}\label{sec:numerical-experiments}
In this section, we conducted numerical experiments to examine the performance of our proposed algorithm.
Note that KL exponents for problems in numerical experiments are unknown.
All numerical experiments were performed in Python 3.9 on a MacBook Air with an Apple M2 and 16GB LPDDR5 memory. 
\subsection{The \texorpdfstring{$\ell_p$}{lp} Regularized Least Squares Problem}\label{sec:the-lp-regularized-least-squares-problem}
We consider the sparse $\ell_p$ (with $p$ slightly larger than $1$) regularized least squares problem~\cite{Chung2019-kg,Wen2016-fg}:
\begin{align}
    \min_{\x\in\R^n}\quad \frac{1}{2}\|\A\x - \b\|^2 +\frac{\theta_p}{p}\|\x\|_p^p,\label{prob:lp}
\end{align}
where $\A\in\R^{m\times n},\b\in\R^m$, and $\theta_p > 0$.
Let $g\equiv0$,
\begin{align}
    f(\x) = \frac{1}{2}\|\A\x - \b\|^2 +\frac{\theta_p}{p}\|\x\|_p^p,\quad\text{and}\quad\phi(\x) = \frac{1}{2}\|\x\|^2 + \frac{1}{p}\|\x\|_p^p.\label{def:lp-f-phi}
\end{align}
Note that $\nabla f$ and $\nabla \phi$ are not Lipschitz continuous for a general $p$ and that when $p > 1$, $\phi$ and $f$ are $\mathcal{C}^1$.
We have $\nabla\phi(\x) = \x + \sgn(\x)|\x|^{p-1}$.

\begin{proposition}
    Let $f$ and $\phi$ be as defined above.
    Then, for any $L > 0$ satisfying
    \begin{align}
        L \geq \lambda_{\max}(\A^{\mathsf{T}}\A) + \theta_p, \label{ineq:l-smad-lp}
    \end{align}
    the functions $L\phi - f$ and $L\phi + f$ are convex on $\R^n$.
    Therefore, the pair $(f,\phi)$ is $L$-smad.
\end{proposition}
\begin{proof}
    Because $f$ and $\phi$ are convex, $L\phi + f$ is convex.
    We have 
    \begin{align*}
        L\phi(\x) - f(\x) &= L\left(\frac{1}{2}\|\x\|^2 + \frac{1}{p}\|\x\|_p^p\right) - \frac{1}{2}\|\A\x - \b\|^2 - \frac{\theta_p}{p}\|\x\|_p^p\\
        &=\frac{L}{2}\|\x\|^2 - \frac{1}{2}\|\A\x - \b\|^2 + \frac{L-\theta_p}{p}\|\x\|_p^p
    \end{align*}
    which is convex for any $\x\in\R^n$ because $\frac{L}{2}\|\x\|^2 - \frac{1}{2}\|\A\x - \b\|^2$ is convex due to Lipschitz continuity of the gradient of $\frac{1}{2}\|\A\x - \b\|^2$  with $\lambda_{\max}(\A^{\mathsf{T}}\A)$ (see, for example,~\cite[Example 2.2]{Beck2009-kr}) and $\frac{L-\theta_p}{p}\|\x\|_p^p$ is convex due to $L-\theta_p\geq0$.
    Therefore, the pair $(f,\phi)$ is $L$-smad.
\end{proof}
The optimality condition of~\eqref{def:bpg-mappring} is equivalent to a $(p-1)$th polynomial equation.
The BPG mapping $\mathcal{T}_\lambda(\x)$ is not computed in a closed-form expression for a general $p > 1$. 

Using the above $\phi$, we obtain Subproblem~\eqref{subprob:abpg} in a closed form.
We have 
\begin{align*}
    \nabla^2\phi(\bm{x}) &= \bm{I} + (p-1)\diag(|\bm{x}|^{p-2}).
\end{align*}
Because the inverse matrix of $\nabla^2\phi(\bm{x})$ is also a diagonal matrix, Subproblem~\eqref{subprob:abpg} is expressed in a closed form.
Note that $\nabla^2\phi(\bm{x})$ is not continuous on $x_i = 0$ for some $i = 1,\ldots, n$ with $p \in(1,2)$, 
while it is continuous on $\R^n$ with $p\geq 2$.

\begin{remark}\label{remark:lp}
    While we handled \eqref{prob:lp} as a case of $g\equiv 0$, 
    we can obtain a closed-form solution of~\eqref{subprob:abpg} even when $g\not\equiv0$, e.g.,
    $g(\x) = \theta_1\|\x\|_1$ with $\theta_1 > 0$. For $p > 1$, as a subproblem of \eqref{prob:lp} with the additional term $g(\x) = \theta_1\|\x\|_1$, we obtain
    \begin{align*}
        \d^k &= \left(\argmin_{x\in\R}\left\{v_i x + \theta_1|x| + \frac{1}{2\lambda}\left(1+(p-1)|x^k_i|^{p-2}\right)(x-x^k_i)^2\right\}\right)_{i=1}^n - \x^k\\
        &=\left(\soft_{\theta_1 s_i}\left(x^k_i - s_iv_i\right)\right)_{i=1}^n - \x^k,
    \end{align*}
    where $v_i$ is the $i$th element of $\nabla f(\x^k)$, $\soft_s(\cdot)$ is a soft-thresholding operator with $s > 0$, and $s_i = \frac{\lambda}{1+(p-1)|x^k_i|^{p-2}}$.
    Since $s_i = \frac{\lambda}{1+(p-1)|x^k_i|^{p-2}} = \frac{\lambda|x^k_i|^{2-p}}{|x^k_i|^{2-p} + p -1}$ for $p\in(1,2)$, $s_i = 0$ when $x_i^k = 0$.
\end{remark}

We compared ABPG with the proximal gradient algorithm (PG), the PG with line search (PGL), and the regularized Newton method (RN).
We set $\alpha = 0.99$ and $\eta = 0.9$ for ABPG.
Although $\nabla f$ is not Lipschitz continuous, we used $1/L$ as stepsizes of PG and as an initial stepsize of PGL using the $L$ of the $L$-smad property given by~\eqref{ineq:l-smad-lp}.
This implies that PG would lose convergence guarantees based on the global-Lipschitz-gradient assumption.
We used PGL with the backtracking procedure~\cite[p.283]{Beck2017-qc}.
We used RN with the fixed constants of regularization.
Note that RN is equivalent to ABPG with $\phi = f + \frac{\kappa}{2}\|\cdot\|^2$. The parameter $\kappa$ was fixed to $10^{-5}> 0$ in this experiment.
We generated the initial point $\x^0\in\R^n$ from i.i.d. Gaussian distribution.
The maximum iteration is 1000 and the termination condition is $\|\x^k - \x^{k-1}\|\leq10^{-6}$.

\begin{figure}[!tbp]
    \begin{minipage}[b]{0.49\linewidth}
    \centering
    \includegraphics[width=\textwidth]{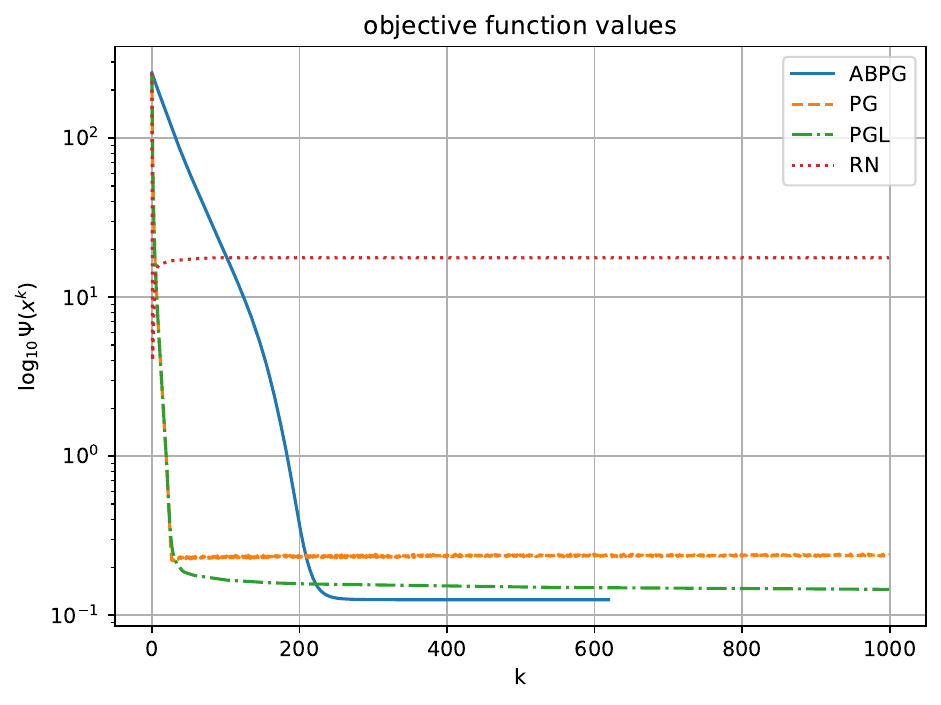}
    \subcaption{Objective function values ($p=1.1$)}
    \label{fig:lp-reg-obj-p=1.1}
    \end{minipage}
    \hfill
    \begin{minipage}[b]{0.49\linewidth}
    \centering
    \includegraphics[width=\textwidth]{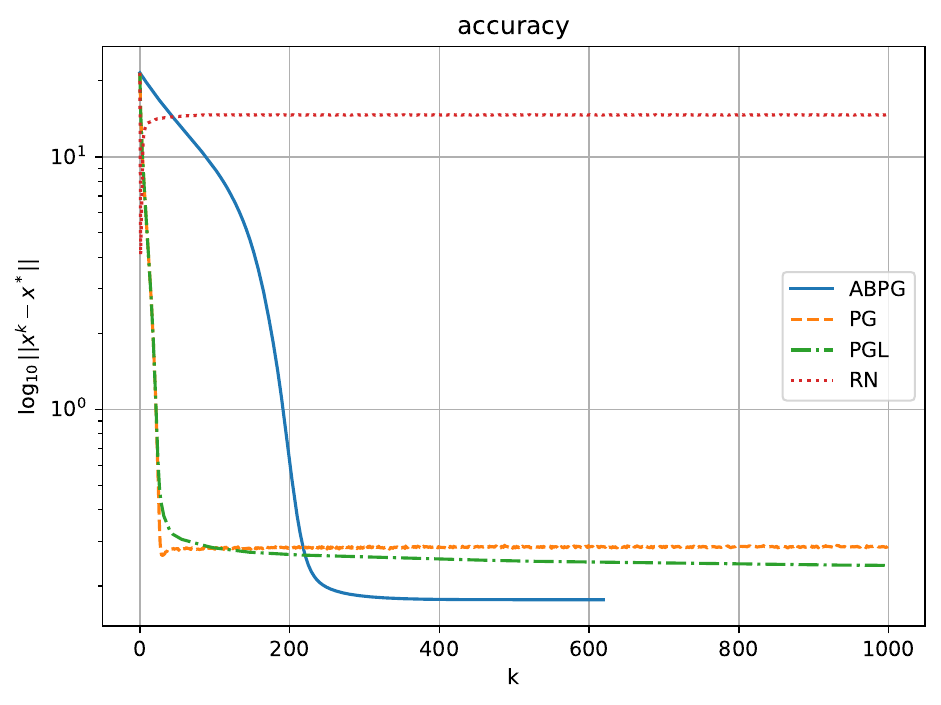}
    \subcaption{Accuracy ($p=1.1$)}
    \label{fig:lp-reg-acc-p=1.1}
    \end{minipage}
    \begin{minipage}[b]{0.49\linewidth}
    \centering
    \includegraphics[width=\textwidth]{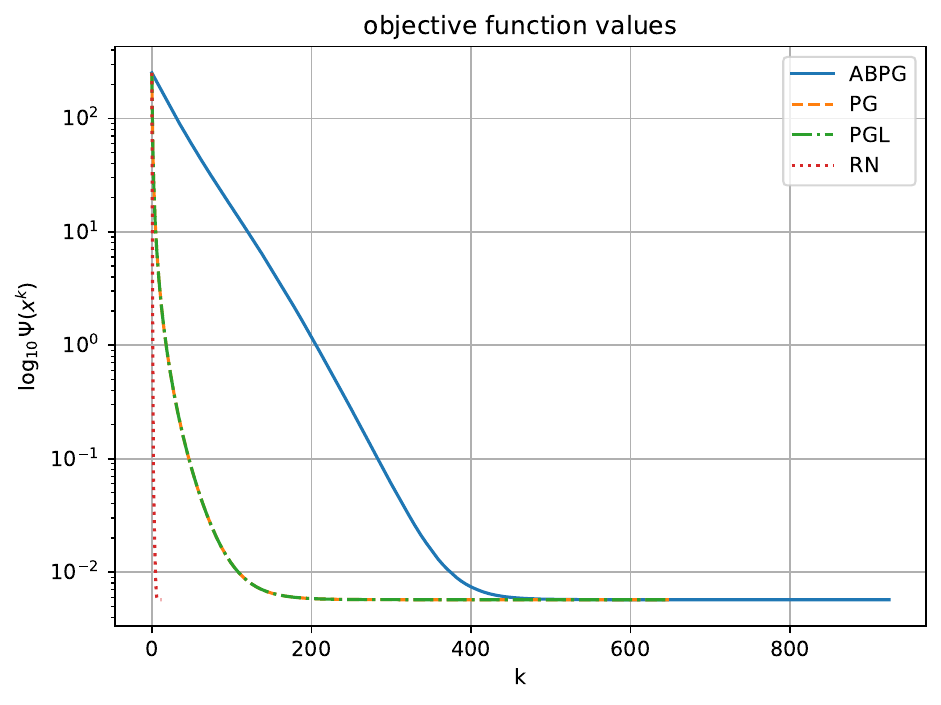}
    \subcaption{Objective function values ($p=3.0$)}
    \label{fig:lp-reg-obj-p=3.0}
    \end{minipage}
    \hfill
    \begin{minipage}[b]{0.49\linewidth}
    \centering
    \includegraphics[width=\textwidth]{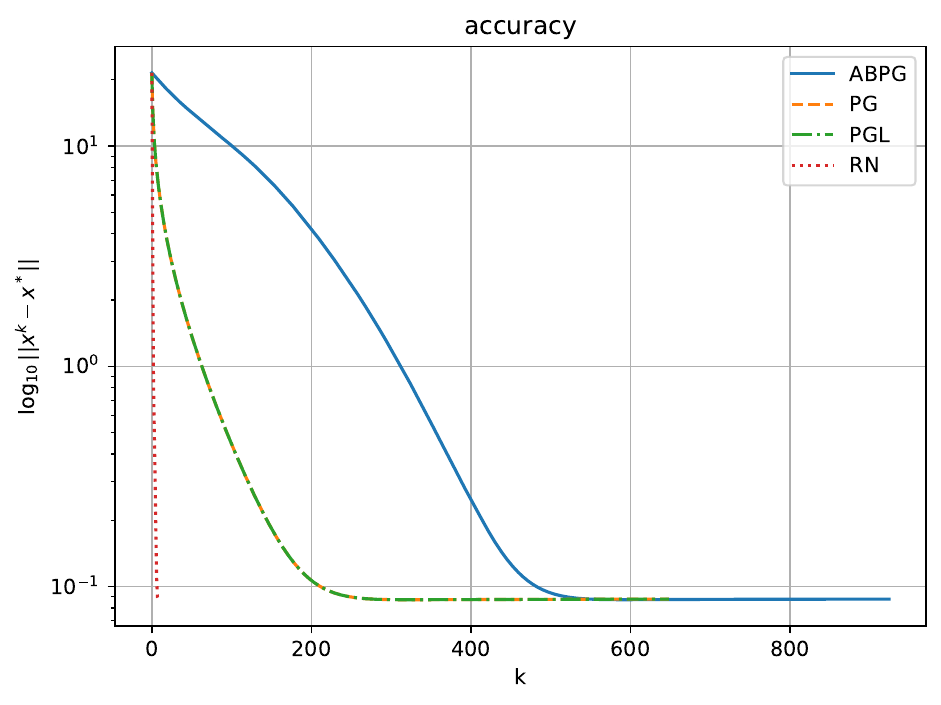}
    \subcaption{Accuracy ($p=3.0$)}
    \label{fig:lp-reg-acc-p=3.0}
    \end{minipage}
    \caption{Comparison with ABPG (blue), PG (orange), PGL (green), and RN (red) on the $\ell_p$ regularized least squares problem~\eqref{prob:lp}.}
    \label{fig:lp-reg}
\end{figure}

\begin{figure}[!tbp]
    \begin{minipage}[b]{0.49\linewidth}
    \centering
    \includegraphics[width=\textwidth]{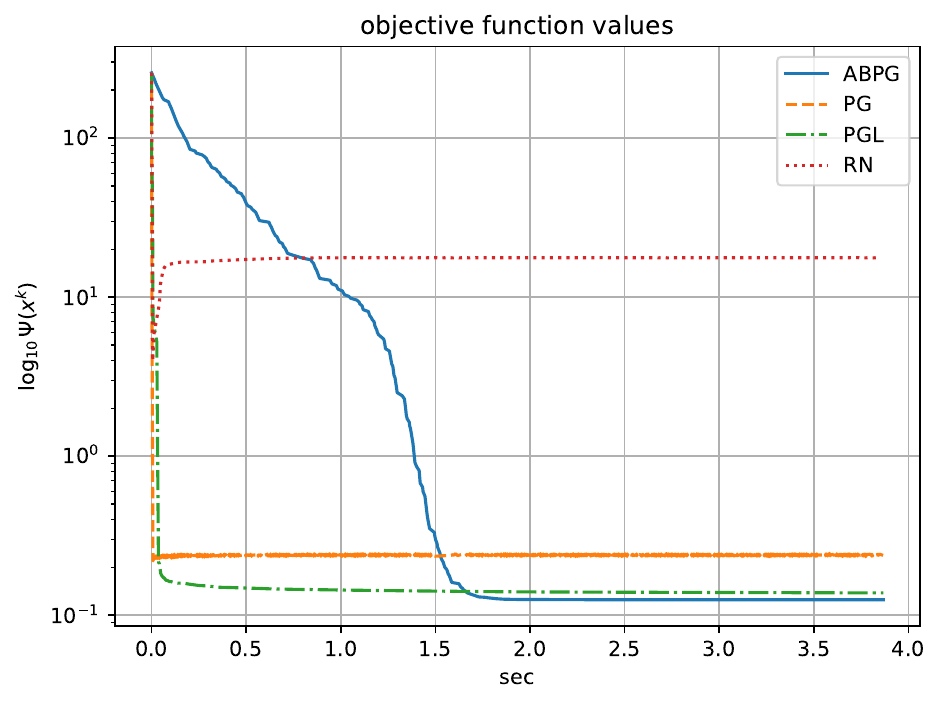}
    \subcaption{Objective function values ($p=1.1$)}
    \label{fig:lp-reg-obj-p=1.1-time}
    \end{minipage}
    \hfill
    \begin{minipage}[b]{0.49\linewidth}
    \centering
    \includegraphics[width=\textwidth]{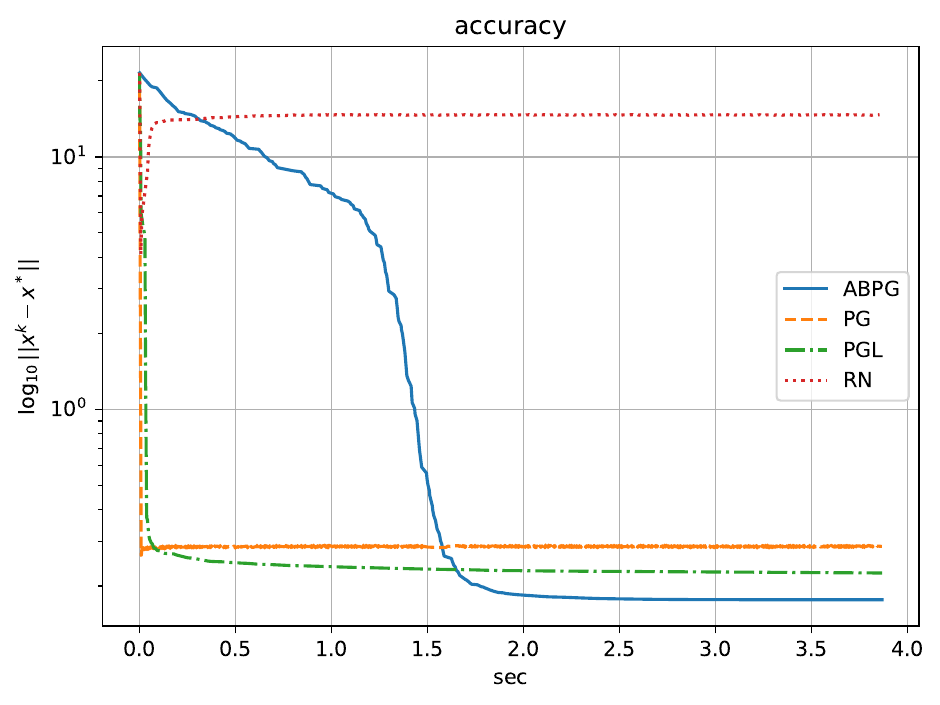}
    \subcaption{Accuracy ($p=1.1$)}
    \label{fig:lp-reg-acc-p=1.1-time}
    \end{minipage}
    \begin{minipage}[b]{0.49\linewidth}
    \centering
    \includegraphics[width=\textwidth]{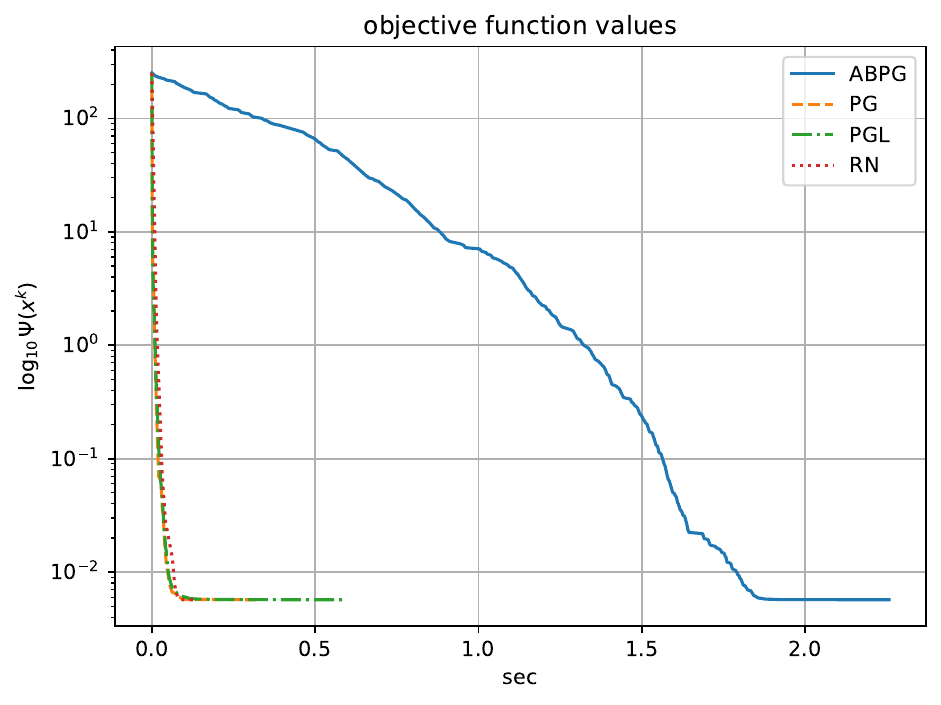}
    \subcaption{Objective function values ($p=3.0$)}
    \label{fig:lp-reg-obj-p=3.0-time}
    \end{minipage}
    \hfill
    \begin{minipage}[b]{0.49\linewidth}
    \centering
    \includegraphics[width=\textwidth]{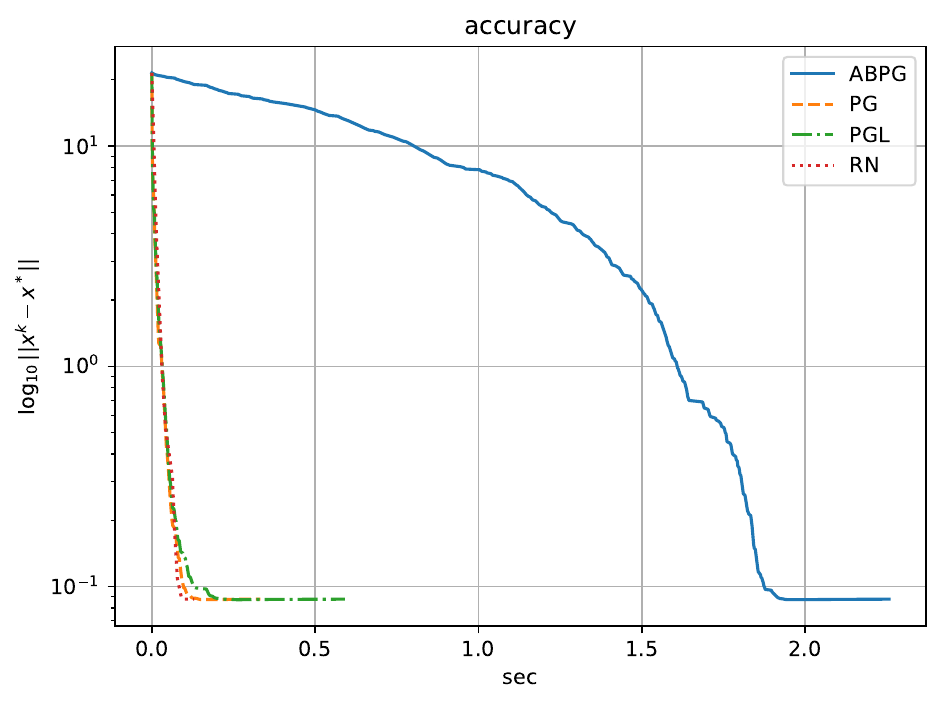}
    \subcaption{Accuracy ($p=3.0$)}
    \label{fig:lp-reg-acc-p=3.0-time}
    \end{minipage}
    \caption{Comparison with ABPG (blue), PG (orange), PGL (green), and RN (red) on the $\ell_p$ regularized least squares problem~\eqref{prob:lp}.}
    \label{fig:lp-reg-time}
\end{figure}

We generated the measurement $\A\in\R^{m\times n}$ and the ground truth $\x^*$ from i.i.d. Gaussian distribution, which has $5\%$ nonzero elements.
We set $\b = \A\x^*$.
For $(n,m) = (500, 800)$, $p=1.1, 3.0$, and $\theta_p = 0.05$, Figure~\ref{fig:lp-reg} shows the objective function values $\Psi(\x^k)$ and the accuracy $\|\x^k - \x^*\|$ at each iteration on a logarithmic scale.
Figures~\ref{fig:lp-reg-obj-p=1.1} and~\ref{fig:lp-reg-acc-p=1.1} show the objective function values and the accuracy at each $k$th iteration, and Figures~\ref{fig:lp-reg-obj-p=1.1-time}, and~\ref{fig:lp-reg-acc-p=1.1-time} show these on time axis.
When PG, PGL, and RN did not stop in 1000 iterations, we plotted their performance until the termination time of ABPG.
Figures~\ref{fig:lp-reg-obj-p=1.1},~\ref{fig:lp-reg-acc-p=1.1},~\ref{fig:lp-reg-obj-p=1.1-time}, and~\ref{fig:lp-reg-acc-p=1.1-time} show that ABPG outperforms other algorithms when $p = 1.1$.
Only ABPG stopped within 1000 iterations (around 600 iterations).
PG and RN did not decrease the objective function value at every iteration; the reason is probably that $\nabla f$ is not globally Lipschitz continuous or even locally Lipschitz continuous, more specifically, 
for $p = 1.1$, $\|\x\|_p^p$ does not have any Lipschitz continuous gradients on $(-1,1)^n$.
Stepsizes of PGL were small because PGL requires many line search procedures due to lacking the global Lipschitz continuity of $\nabla f$.
For $p=3.0$, Figures~\ref{fig:lp-reg-obj-p=3.0} and~\ref{fig:lp-reg-acc-p=3.0} show that ABPG is outperformed by other algorithms.
All algorithms stopped within 1000 iterations.
The objective function value and the accuracy of all algorithms at the termination were almost the same.
In this case, the backtracking procedure of PGL was not used.
For $p \geq 2$, $\|\x\|_p^p$ has Lipschitz continuous gradients on $(-1,1)^n$ and the backtracking procedure of PGL is not needed on $(-1,1)^n$.
Therefore, PG and PGL outperformed ABPG.

\begin{table}[!tbp]
\centering
\caption{Performance results over 50 random instances for $m = 1000$.
Numbers in boldface represent the best value among algorithms.
This is also used for Table~\ref{tab:lp-2000}.
``---" means that its value is too large.}
\label{tab:lp-1000}
\begin{tabular}{lllrrr}
\toprule
$m$ & $n$ & Algorithm &  Iteration &            Obj. &            Acc. \\\midrule
1000 & 100  & ABPG &   \textbf{554} &        \textbf{0.07502} &        \textbf{0.09667} \\
     &      & PG &  1000 &        0.12554 &        0.17612 \\
     &      & PGL &   970 &        0.07863 &        0.12067 \\
     &      & RN &  1000 &        0.23388 &        0.39102 \\
     & 200  & ABPG &   \textbf{580} &        \textbf{0.09696} &        \textbf{0.12952} \\
     &      & PG &  1000 &        0.16969 &        0.21781 \\
     &      & PGL &  1000 &        0.10451 &        0.16348 \\
     &      & RN &  1000 &        0.68147 &        0.94434 \\
     & 500  & ABPG &   \textbf{619} &        \textbf{0.13928} &        \textbf{0.19568} \\
     &      & PG &  1000 &        0.26540 &        0.31184 \\
     &      & PGL &  1000 &        0.15927 &        0.26231 \\
     &      & RN &  1000 &        7.66003 &        6.69885 \\
     & 1000 & ABPG &   \textbf{652} &        \textbf{0.17662} &        \textbf{0.26180} \\
     &      & PG &  1000 &        0.34935 &        0.38783 \\
     &      & PGL &  1000 &        0.22393 &        0.40126 \\
     &      & RN &  1000 & --- & ---\\
\bottomrule
\end{tabular}
\end{table}
\begin{table}[!htbp]
\caption{Performance results over 50 random instances for $m = 2000$.}
\label{tab:lp-2000}
\centering
\begin{tabular}{lllrrr}
\toprule
$m$ & $n$ & Algorithm &  Iteration &            Obj. &            Acc. \\\midrule
2000 & 100  & ABPG &   \textbf{558} &        \textbf{0.07333} &        \textbf{0.09635} \\
     &      & PG &  1000 &        0.13971 &        0.20044 \\
     &      & PGL &   939 &        0.07660 &        0.11364 \\
     &      & RN &  1000 &        0.20294 &        0.32129 \\
     & 200  & ABPG &   \textbf{575} &        \textbf{0.09443} &        \textbf{0.12819} \\
     &      & PG &  1000 &        0.19450 &        0.24123 \\
     &      & PGL &  1000 &        0.10177 &        0.15410 \\
     &      & RN &  1000 &        0.43165 &        0.57237 \\
     & 500  & ABPG &   \textbf{602} &        \textbf{0.13678} &        \textbf{0.19227} \\
     &      & PG &  1000 &        0.29906 &        0.32302 \\
     &      & PGL &  1000 &        0.15525 &        0.23809 \\
     &      & RN &  1000 &        2.06468 &        1.90547 \\
     & 1000 & ABPG &   \textbf{631} &        \textbf{0.17750} &        \textbf{0.25676} \\
     &      & PG &  1000 &        0.42798 &        0.41943 \\
     &      & PGL &  1000 &        0.21516 &        0.33322 \\
     &      & RN &  1000 &       15.67614 &        9.69705 \\
\bottomrule
\end{tabular}
\end{table}

Moreover, Tables~\ref{tab:lp-1000} and~\ref{tab:lp-2000} show average performance over 50 random instances for $p =1.1, m = 1000, 2000$, and $n = 100, 200, 500, 1000$.
Only ABPG stopped within 1000 iterations and outperformed other algorithms.


\subsection{The \texorpdfstring{$\ell_p$}{lp} Regularized Least Squares Problem with Equality Constraints}\label{sec:the-lp-regularized-least-squares-problem-with-constraints}
We next consider the $\ell_p$ regularized least squares problem with an equality constraint~\cite{Peng_Ni2008-hx}:
\begin{align}
    \min_{\x\in \R^n}\quad \frac{1}{2}\|\A\x - \b\|^2 + \frac{\theta_p}{p}\|\x\|_p^p + \delta_S(\x),\label{prob:lp-constrained}
\end{align}
where $\delta_S$ denotes the indicator function of $S = \{\x\in\R^n\mid\a^{\mathsf{T}}\x = \gamma\}$ with $\a\in\R^n$ and $\gamma\in\R$.
Let $g(\x) = \delta_S(\x)$.
We use the same definition as~\eqref{def:lp-f-phi} for $f$ and $\phi$.

Let us consider Subproblem~\eqref{subprob:abpg} for solving~\eqref{prob:lp-constrained}.
The feasible points $\x^k$ need to satisfy $\a^{\mathsf{T}}\x^k = \gamma$, which implies that $\a^{\mathsf{T}}\x^k = \a^{\mathsf{T}}(\x^{k-1}+t_k\d^k) = \gamma + t_k\a^{\mathsf{T}}\d^{k-1} = \gamma$.
Therefore, the directions $\d^k$ satisfy $\a^{\mathsf{T}}\d^k = 0$ for all $k\geq0$.
We have the optimality condition of~\eqref{subprob:abpg},
\begin{align*}
    \left[\begin{array}{cc}
        \nabla^2\phi(\bm{x}^k) & \bm{a} \\
        \bm{a}^{\mathsf{T}} & 0
    \end{array}\right]
    \left[\begin{array}{c}
    \bm{d}^k\\
    \mu
    \end{array}
    \right] &= \left[\begin{array}{c}
        -\lambda\nabla f(\bm{x}^k)\\
        0
    \end{array}
    \right],
\end{align*}
where $\mu\in\R$.
Let $\delta = -\bm{a}^{\mathsf{T}}\nabla^2\phi(\bm{x}^k)^{-1}\bm{a}\neq0$.
We solve the linear system and have
\begin{align*}
    \left[\begin{array}{c}
    \bm{d}^k\\
    \mu
    \end{array}
    \right] &= \left[\begin{array}{cc}
        \nabla^2\phi(\bm{x}^k)^{-1} + \delta^{-1}\u\u^{\mathsf{T}} & -\delta^{-1}\u \\
        -\delta^{-1}\u^{\mathsf{T}} & \delta^{-1}
    \end{array}\right]\left[\begin{array}{c}
        -\lambda\nabla f(\bm{x}^k)\\
        0
    \end{array}
    \right]\\
    &= \left[\begin{array}{c}
        -\lambda\nabla^2\phi(\bm{x}^k)^{-1}\nabla f(\bm{x}^k) - \delta^{-1}\lambda\u\u^{\mathsf{T}}\nabla f(\bm{x}^k)\\
        \delta^{-1}\lambda\u^{\mathsf{T}}\nabla f(\bm{x}^k)
    \end{array}\right],
\end{align*}
where $\u = \nabla^2\phi(\bm{x}^k)^{-1}\bm{a}$.
While we consider the case of one linear equality constraint, the above result can be extended to the case of multiple linear equality constraints.
Set $\nabla^2\phi(\bm{x}^k) = \I$ for the proximal gradient method and $\nabla^2\phi(\bm{x}^k) = \kappa\I + \nabla f(\x^k)$ for the regularized Newton method with a fixed parameter $\kappa > 0$.

We generated $\x^0 \in S$ and set $\a = \one$ and $\gamma = 1$.
The other setting for parameters and the termination condition is the same as Section~\ref{sec:the-lp-regularized-least-squares-problem}.
For $(n,m) = (500, 800), p = 1.1$ and $\theta_p = 0.05$, Figures~\ref{fig:lp-reg-const} and~\ref{fig:lp-reg-const-time} show that the objective function values $\Psi(\x^k)$ and the accuracy $\|\x^k - \x^*\|$ for~\eqref{prob:lp-constrained} on a logarithmic scale.
ABPG outperformed other algorithms on the objective function values and the accuracy (Figures~\ref{fig:lp-reg-obj-const-p=1.1} and~\ref{fig:lp-reg-acc-const-p=1.1} for the iteration axis and Figures~\ref{fig:lp-reg-obj-const-p=1.1-time} and~\ref{fig:lp-reg-acc-const-p=1.1-time} for the time axis).
PG, PGL, and RN did not stop within 1000 iterations.
The reason for this is the same reason for~\eqref{prob:lp}.
PG did not decrease the objective function value due to lacking the global Lipschitz continuity for $\nabla f$.
The line search procedure of PGL did not work well.
RN also did not work well because of ill-conditioned inverse matrices.

\begin{figure}[!tbp]
    \begin{minipage}[b]{0.49\linewidth}
    \centering
    \includegraphics[width=\textwidth]{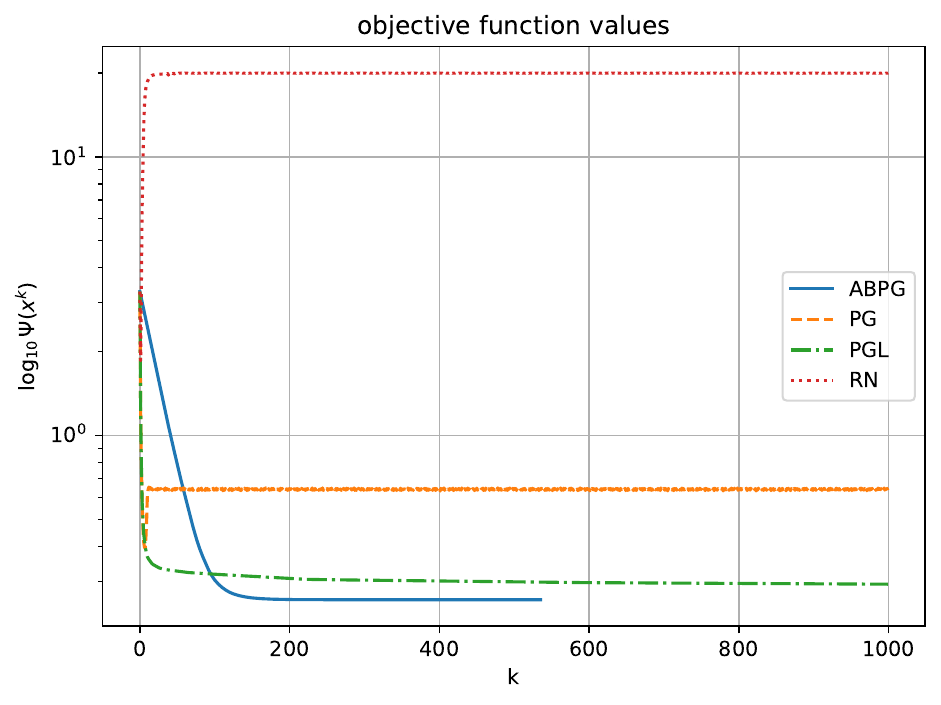}
    \subcaption{Objective function values ($p=1.1$)}
    \label{fig:lp-reg-obj-const-p=1.1}
    \end{minipage}
    \hfill
    \begin{minipage}[b]{0.49\linewidth}
    \centering
    \includegraphics[width=\textwidth]{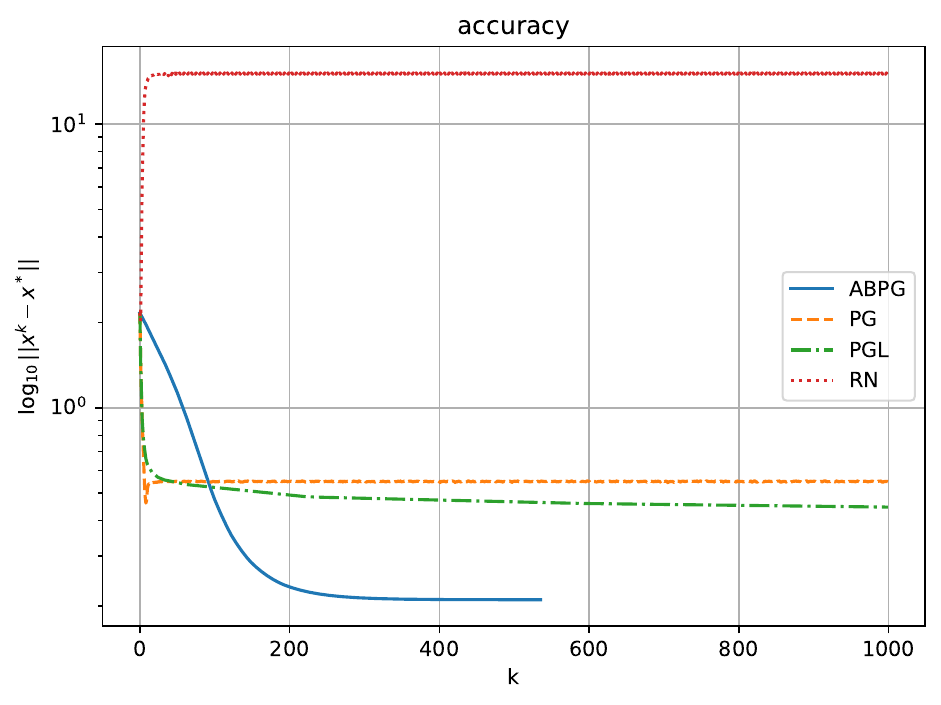}
    \subcaption{Accuracy ($p=1.1$)}
    \label{fig:lp-reg-acc-const-p=1.1}
    \end{minipage}
    \caption{Comparison with ABPG (blue), PG (orange), PGL (green), and RN (red) on the $\ell_p$ regularized least squares problem with an equality constraint~\eqref{prob:lp-constrained}.}
    \label{fig:lp-reg-const}
\end{figure}

\begin{figure}[!tbp]
    \begin{minipage}[b]{0.49\linewidth}
    \centering
    \includegraphics[width=\textwidth]{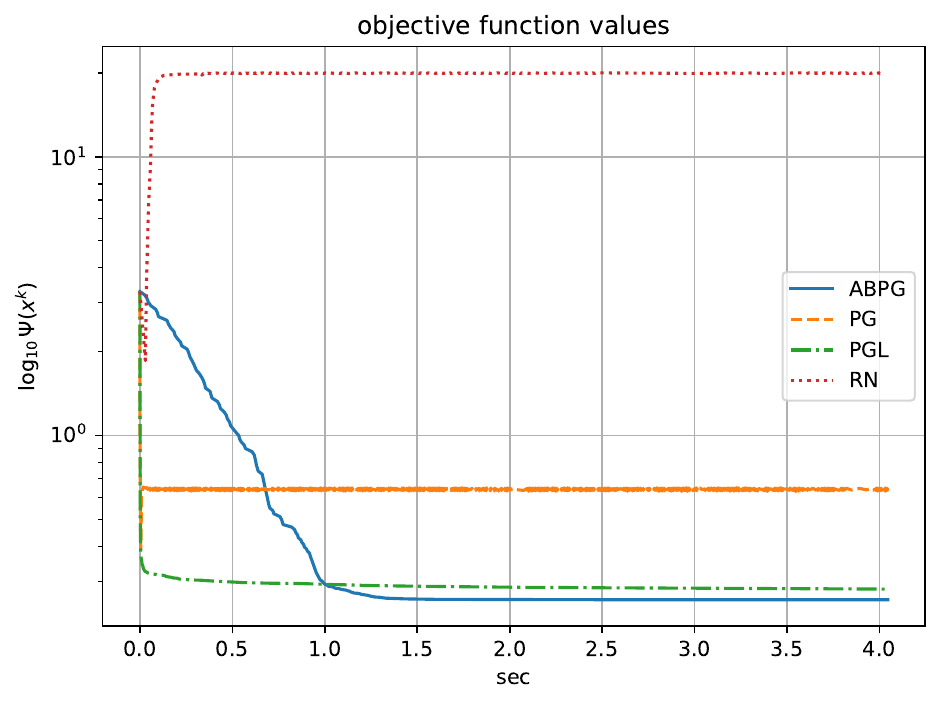}
    \subcaption{Objective function values ($p=1.1$)}
    \label{fig:lp-reg-obj-const-p=1.1-time}
    \end{minipage}
    \hfill
    \begin{minipage}[b]{0.49\linewidth}
    \centering
    \includegraphics[width=\textwidth]{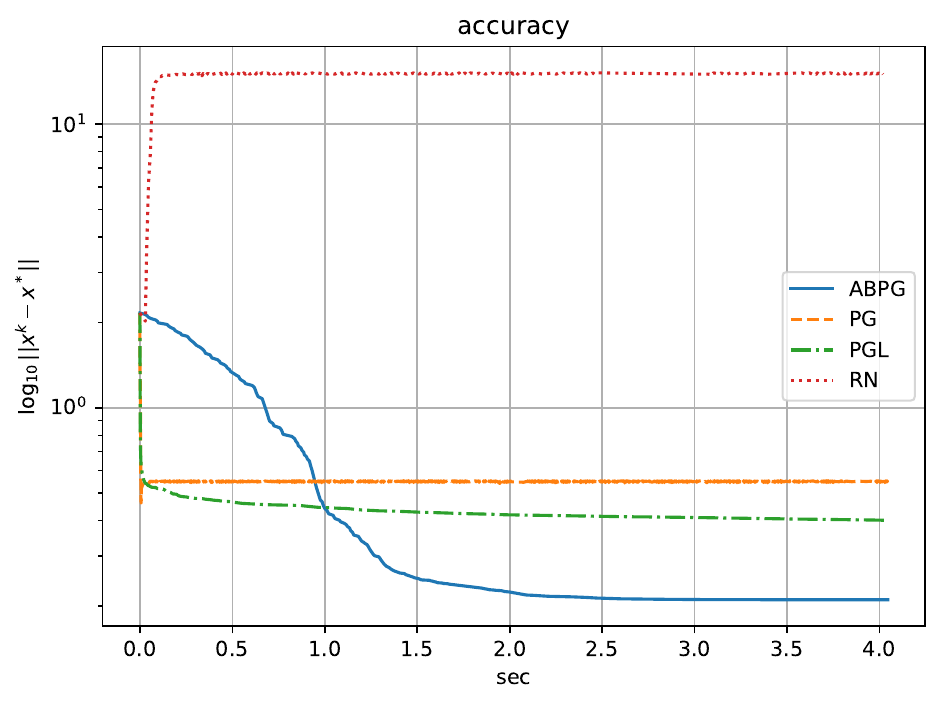}
    \subcaption{Accuracy ($p=1.1$)}
    \label{fig:lp-reg-acc-const-p=1.1-time}
    \end{minipage}
    \caption{Comparison with ABPG (blue), PG (orange), PGL (green), and RN (red) on the $\ell_p$ regularized least squares problem with an equality constraint~\eqref{prob:lp-constrained}.}
    \label{fig:lp-reg-const-time}
\end{figure}

\subsection{The \texorpdfstring{$\ell_p$}{lp} Loss Problem}\label{sec:lp-loss}
We consider the $\ell_p$ loss problem~\cite{Maddison2021-nt}:
\begin{align}
    \min_{\x\in \R^n}\quad \frac{1}{p}\|\A\x - \b\|_p^p,\label{prob:lp-loss}
\end{align}
where $\A\in\R^{m\times n},\b\in\R^m$, and $p > 1$.
Let $f(\x) = \frac{1}{p}\|\A\x - \b\|_p^p$ and $g\equiv0$.
We have 
\begin{align*}
    \nabla f(\bm{x}) &= \sum_{i=1}^m|\bm{a}_i^{\mathsf{T}}\bm{x} - b_i|^{p-1}\sgn(\bm{a}_i^{\mathsf{T}}\bm{x} - b_i)\bm{a}_i,
\end{align*}
which is continuous with $p > 1$.
We also have $\nabla^2 f(\bm{x}) = (p-1)\sum_{i=1}^m|\bm{a}_i^{\mathsf{T}}\bm{x} - b_i|^{p-2}\bm{a}_i\bm{a}_i^{\mathsf{T}}$.
Using $\phi = f + \frac{1}{2}\|\cdot\|^2$, the pair $(f, \phi)$ is 1-smad and $\nabla^2\phi$ is nonsingular. 
Moreover, ABPG is equivalent to RN in this case.
We compare ABPG (RN) with PG and PGL.

\begin{figure}[!tbp]
    \begin{minipage}[b]{0.49\linewidth}
    \centering
    \includegraphics[width=\textwidth]{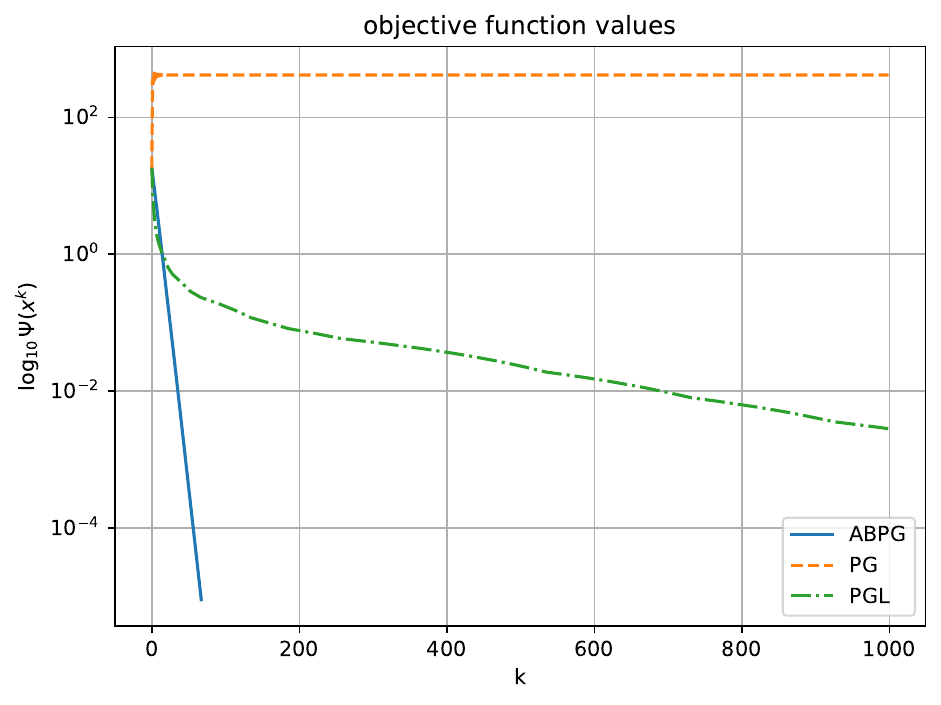}
    \subcaption{Objective function values}
    \label{fig:lp-loss-obj}
    \end{minipage}
    \hfill
    \begin{minipage}[b]{0.49\linewidth}
    \centering
    \includegraphics[width=\textwidth]{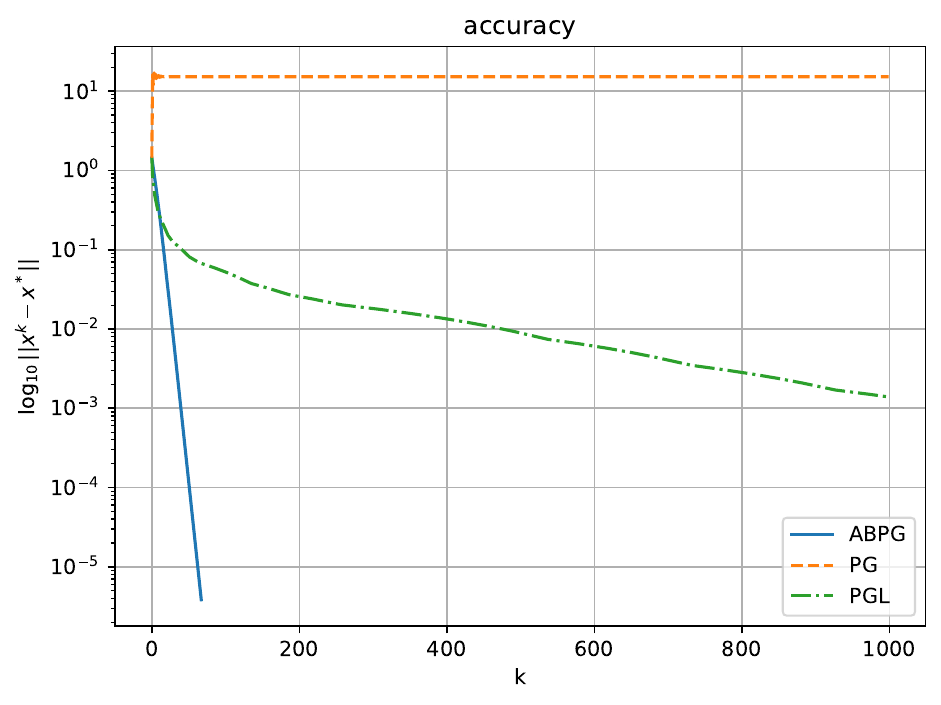}
    \subcaption{Accuracy}
    \label{fig:lp-loss-acc}
    \end{minipage}
    \caption{Comparison with ABPG (blue), PG (orange), and PGL (green) on the $\ell_p$ loss problem~\eqref{prob:lp-loss}.}
    \label{fig:lp-loss}
\end{figure}

\begin{figure}[!tbp]
    \begin{minipage}[b]{0.49\linewidth}
    \centering
    \includegraphics[width=\textwidth]{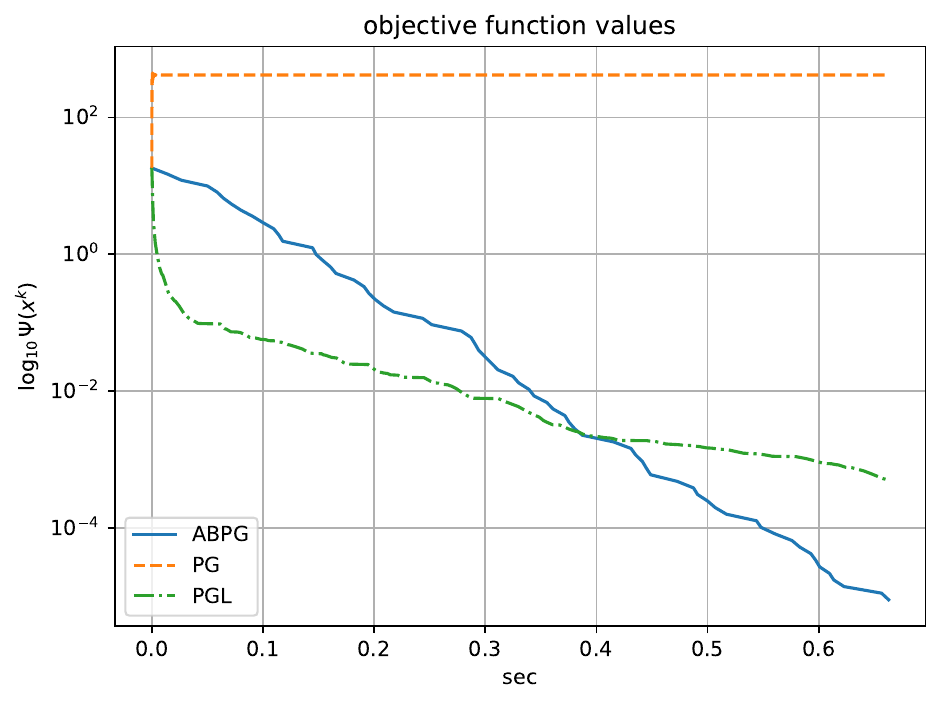}
    \subcaption{Objective function values}
    \label{fig:lp-loss-obj-time}
    \end{minipage}
    \hfill
    \begin{minipage}[b]{0.49\linewidth}
    \centering
    \includegraphics[width=\textwidth]{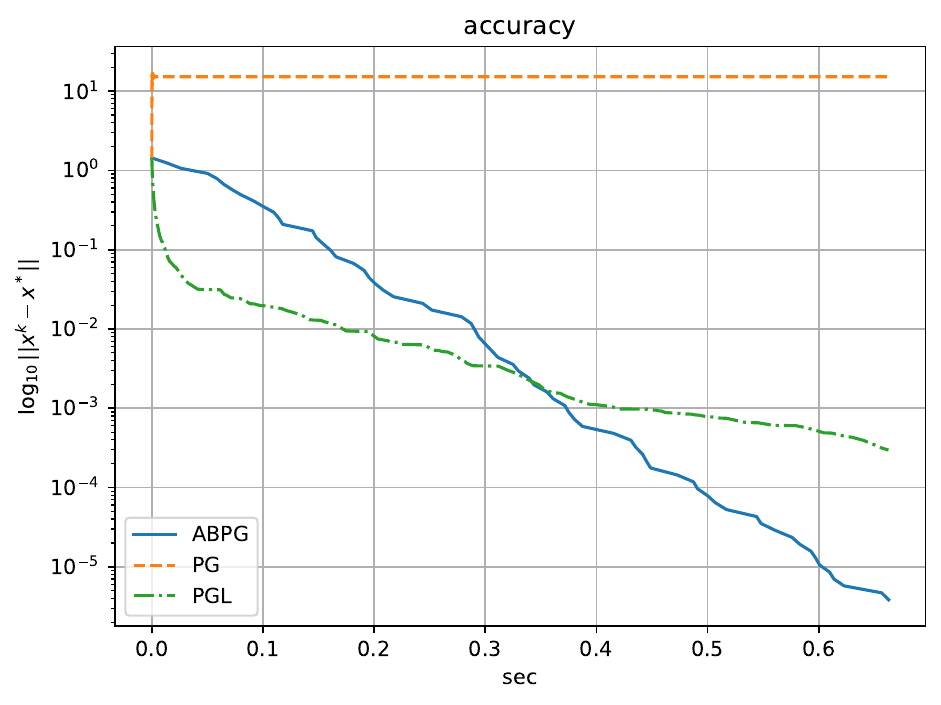}
    \subcaption{Accuracy}
    \label{fig:lp-loss-acc-time}
    \end{minipage}
    \caption{Comparison with ABPG (blue), PG (orange), and PGL (green) on the $\ell_p$ loss problem~\eqref{prob:lp-loss}.}
    \label{fig:lp-loss-time}
\end{figure}

We generated the ground truth $\x^*$ from i.i.d. Gaussian distribution, which has $10\%$ nonzero elements and the initial point $\x^0$ by using the Wirtinger flow initialization~\cite{Candes2015-sd}.
For $(n,m) = (200, 500)$ and $ p = 1.1$, Figure~\ref{fig:lp-loss} shows that the objective function values $\Psi(\x^k)$ and the accuracy $\|\x^k - \x^*\|$ for~\eqref{prob:lp-loss} at each iteration on a logarithmic scale.
On the other hand, Figure~\ref{fig:lp-loss-time} shows these performance results on the time axis.
In this case, PG did not decrease the objective function value (Figures~\ref{fig:lp-loss-obj} and~\ref{fig:lp-loss-obj-time}).
PG and PGL did not stop within 1000 iterations, probably because $\nabla f$ is not globally Lipschitz continuous.
ABPG outperformed PG and PGL.

\subsection{The Nonnegative Linear System on \texorpdfstring{$\R_+^n$}{Rn}}\label{sec:nonnegative-linear-system}
Finally, we consider the nonnegative linear system~\cite{Bauschke2017-hg} and minimize 
\begin{align}
    \min_{\x\in \R_+^n}\quad D_{\textup{KL}}(\A\x,\b) + \theta_1\|\x\|_1, \label{prob:nonnegative-linear}
\end{align}
where
\begin{align}
    D_{\textup{KL}}(\x,\y) = \sum_{i=1}^m \left(x_i\log\frac{x_i}{y_i} + y_i - x_i\right)\label{def:kl-divergence}
\end{align}
is Kullback--Leibler divergence, $\A\in\R_+^{m\times n},\b\in\R_+^m$, and $\theta_1 > 0$.
Let $f(\x) = D_{\textup{KL}}(\A\x,\b)$ and $g(\x) = \theta_1\|\x\|_1$.
Let $\phi_{\textup{BPG}}(\x) = \sum_{i=1}^n x_i\log x_i$ be a kernel generating distance for BPG and $\phi_{\textup{ABPG}}(\x) = \phi_{\textup{BPG}}(\x) + \frac{1}{2}\|\x\|^2$ be that for ABPG.
Therefore, we use $C = \interior\dom\phi_{\textup{BPG}} = \interior\dom\phi_{\textup{ABPG}} = \R_+^n$.
Assuming $\sum_{i=1}^m a_{ij}=1$ for any $j$ (this assumption is general for applications, for example, positron emission tomography~\cite{Vardi1985-de}), we find that $(f,\phi_{\textup{BPG}})$ is 1-smad~\cite{Bauschke2017-hg} and $(f,\phi_{\textup{ABPG}})$ is also 1-smad.

We compare ABPG with BPG to demonstrate the difference in performance by the approximate Bregman distance.
The subproblems of BPG and ABPG for~\eqref{prob:nonnegative-linear} are expressed in a closed form.
We also compare ABPG with PGL.

\begin{figure}[!htbp]
    \begin{minipage}[b]{0.49\linewidth}
    \centering
    \includegraphics[width=\textwidth]{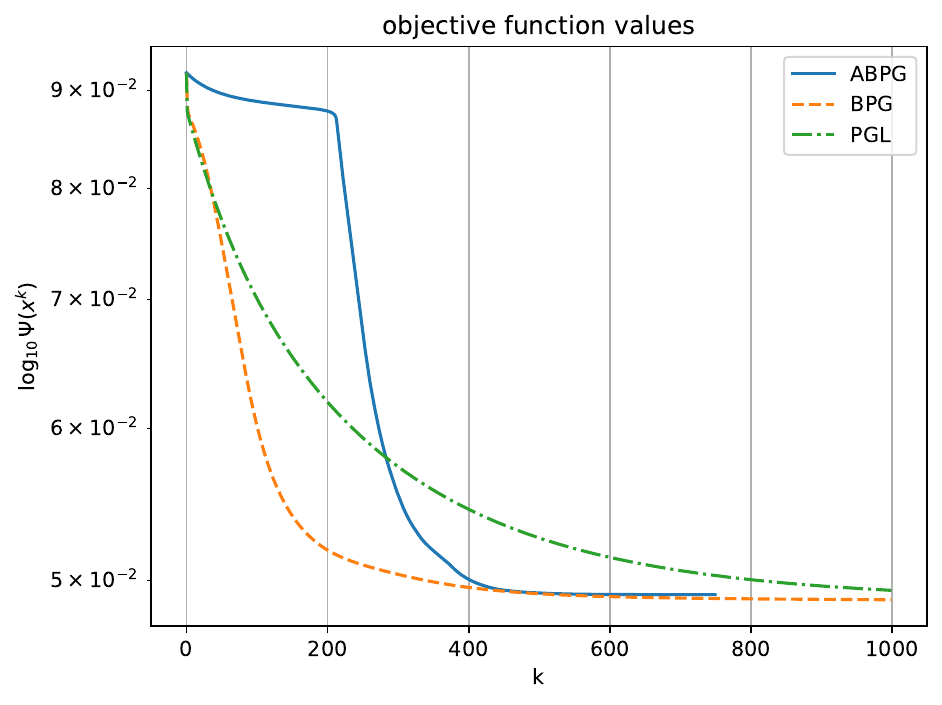}
    \subcaption{Objective function values}
    \label{fig:nonnegative_linear-obj}
    \end{minipage}
    \hfill
    \begin{minipage}[b]{0.49\linewidth}
    \centering
    \includegraphics[width=\textwidth]{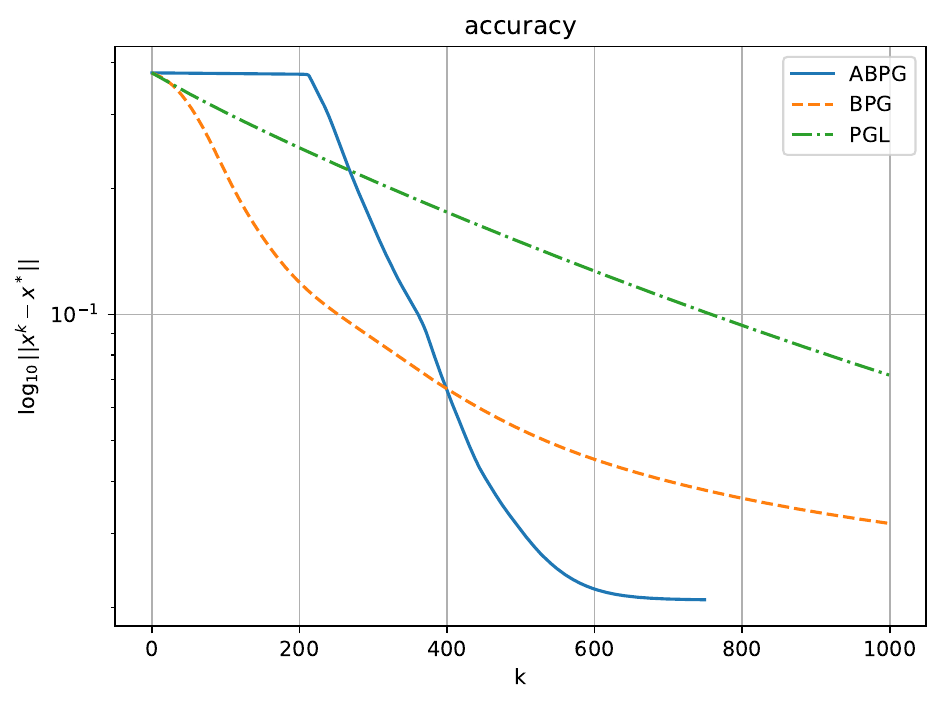}
    \subcaption{Accuracy}
    \label{fig:nonnegative_linear-acc}
    \end{minipage}
    \caption{Comparison with ABPG (blue), BPG (orange), and PGL (green) on the nonnegative linear system inverse problem~\eqref{prob:nonnegative-linear}.}
    \label{fig:nonnegative_linear}
\end{figure}

\begin{figure}[!htbp]
    \begin{minipage}[b]{0.49\linewidth}
    \centering
    \includegraphics[width=\textwidth]{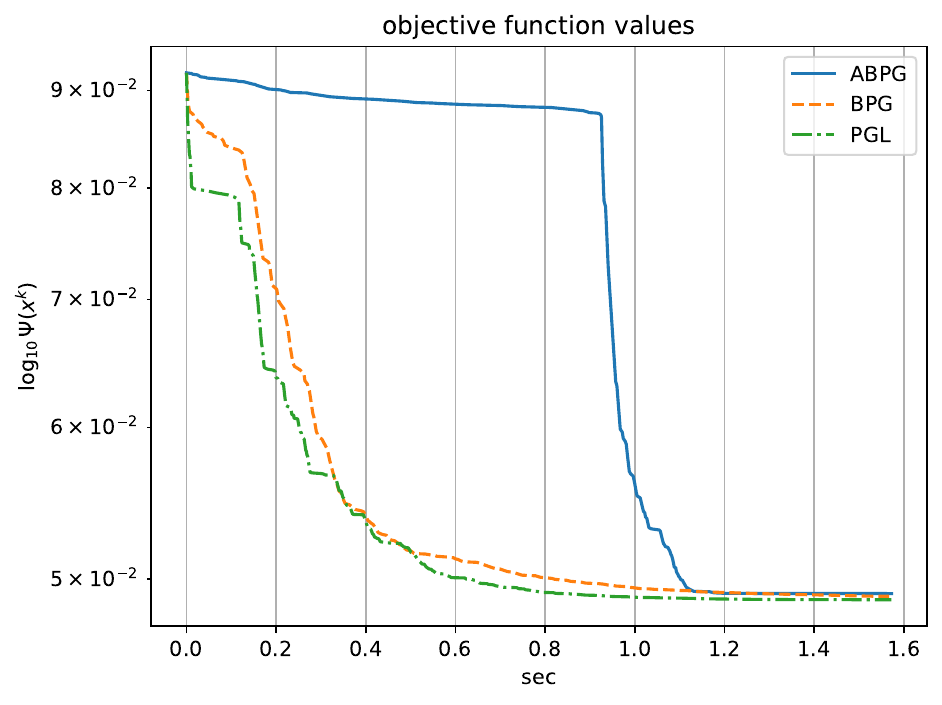}
    \subcaption{Objective function values}
    \label{fig:nonnegative_linear-obj-time}
    \end{minipage}
    \hfill
    \begin{minipage}[b]{0.49\linewidth}
    \centering
    \includegraphics[width=\textwidth]{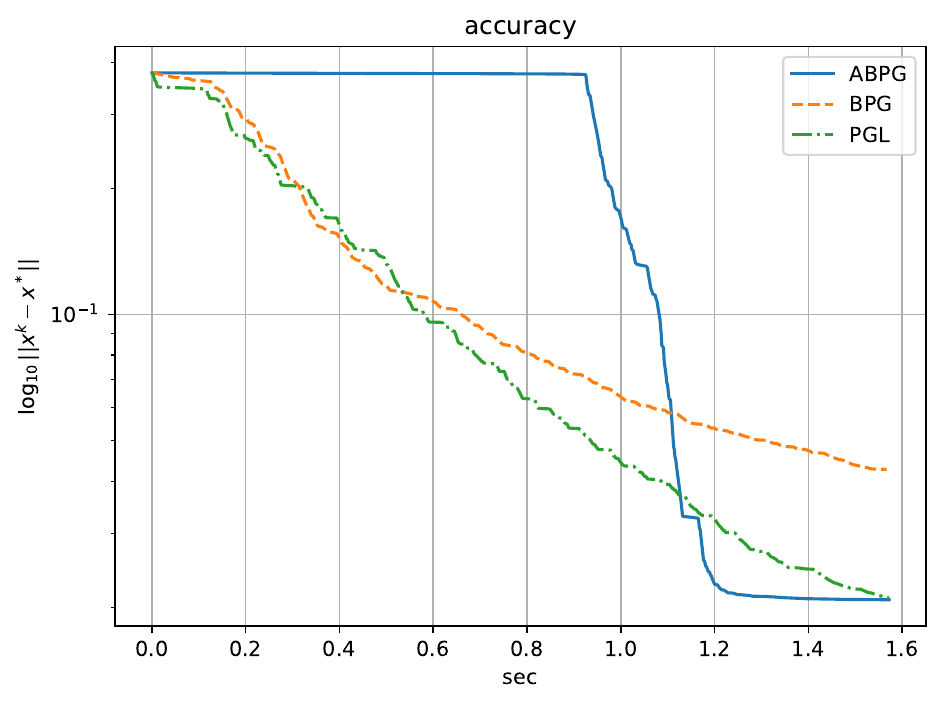}
    \subcaption{Accuracy}
    \label{fig:nonnegative_linear-acc-time}
    \end{minipage}
    \caption{Comparison with ABPG (blue), BPG (orange), and PGL (green) on the nonnegative linear system inverse problem~\eqref{prob:nonnegative-linear}.}
    \label{fig:nonnegative_linear-time}
\end{figure}

We generated the ground truth $\x^*$ from i.i.d. Gaussian distribution, which has $5\%$ nonzero elements.
For $(n,m) = (200, 500)$ and $\theta_1 = 0.05$, Figure~\ref{fig:nonnegative_linear} shows that the objective function values $\Psi(\x^k)$ and the accuracy $\|\x^k - \x^*\|$ for~\eqref{prob:nonnegative-linear} at each iteration on a logarithmic scale.
Figure~\ref{fig:nonnegative_linear-time} shows these performance results on the time axis.
In this setting, while PG and BPG did not stop within 1000 iterations, the objective function value of BPG at the 1000th iteration was the best among these algorithms (Figure~\ref{fig:nonnegative_linear-obj}).
Only ABPG stopped within 1000 iterations and recovered the best point (Figure~\ref{fig:nonnegative_linear-acc}).
Figure~\ref{fig:nonnegative_linear-acc-time} shows PGL (more than the maximum iteration) also recovered a point at which accuracy is almost the same as the point recovered by ABPG.

\section{Conclusion}\label{sec:conclusion}
In this paper, we propose the approximate Bregman proximal gradient algorithm (ABPG) for composite nonconvex optimization problems.
ABPG does not require the global Lipschitz continuity for $\nabla f$.
Its subproblem of ABPG is easy to solve due to the approximate Bregman distance.
In particular, when $\phi$ is strongly convex and $g\equiv0$, the subproblem is always expressed in a closed form.
Moreover, ABPG coincides with existing algorithms by choosing some particular $\phi$.
This implies that we can discuss theoretical aspects of existing algorithms by using the ABPG framework.
We have established the global subsequential convergence with some standard assumptions.
Then, we have guaranteed the global convergence to a stationary point under the KL property when $g\equiv0$.
We also applied ABPG to the $\ell_p$ regularized least squares problems with some cases, the $\ell_p$ loss problem, and the nonnegative linear system.

For large-scale optimization problems, ABPG would be very slow due to the line search procedure.
It is important for future work to develop the acceleration of ABPG.
Recently, some researchers~\cite{Ding2023-jf,Dragomir2022-oh,Yang2022-jr} evaluated the iteration complexity of Bregman-type algorithms.
Because ABPG uses the approximate Bregman distance, it would be challenging to evaluate the iteration complexity of ABPG.
We would like to resolve this matter to evaluate it in the future.
Even if $g\not\equiv0$ is separable, the global convergence would not be established due to the nonsmoothness of $g$.
This might be resolved by improving the line search procedure.

\section*{Declarations}
\textbf{Funding.} This work was supported by JSPS KAKENHI Grant Number JP23K19953; JSPS KAKENHI Grant Number JP23H03351 and JST ERATO Grant Number JPMJER1903.\\
\textbf{Conflict of interest.} The authors have no competing interests to declare that are relevant to the content of this article.\\
\textbf{Data availability.}
The datasets generated during and/or analyzed during the current study are available in the Github repository, \url{https://github.com/ShotaTakahashi/ApproximateBPG}.
\addcontentsline{toc}{section}{References}
\bibliography{main}

\end{document}